\documentclass{m2an}
\usepackage[english]{babel}
\usepackage{amsfonts}
\usepackage{graphicx}
\usepackage{epstopdf}
\usepackage{algorithm}
\usepackage{bm}
\usepackage[hidelinks]{hyperref}
\usepackage{amsthm}
\usepackage{amsmath, amssymb}
\usepackage{multirow}
\usepackage{xcolor}
\numberwithin{equation}{section}

\newcommand{\lang}{\left \langle}
\newcommand{\rang}{\right \rangle}
\newcommand{\myR}{\mathcal{R}}
\renewcommand{\div}{\operatorname{div}}

\newcommand{\half}{\frac{1}{2}}

\newtheorem{theorem}{Theorem}[section]
\newtheorem{lemma}{Lemma}[section]
\newtheorem{corollary}{Corollary}[section]

\usepackage{cleveref}

\begin{document}
\title{A Parameter-Robust Iterative Method for Stokes-Darcy Problems Retaining Local Mass Conservation}
\thanks{We thank the Deutsche Forschungsgemeinschaft (DFG, German Research Foundation) for supporting this work by funding SFB 1313, Project Number 327154368.}
\author{Wietse M. Boon}
\address{University of Stuttgart, 
Department of Hydromechanics and Modelling of Hydrosystems, 
Pfaffenwaldring 61, 70569, Stuttgart \\
Currently at: KTH Royal Institute of Technology, 
Department of Mathematics,
Lindstedtsv\"agen 25, 114 28, Stockholm, \email{wietse@kth.se}}
\date{\today}
\begin{abstract}
	We consider a coupled model of free-flow and porous medium flow, governed by stationary Stokes and Darcy flow, respectively. The coupling between the two systems is enforced by introducing a single variable representing the normal flux across the interface. The problem is reduced to a system concerning only the interface flux variable, which is shown to be well-posed in appropriately weighted norms. An iterative solution scheme is then proposed to solve the reduced problem such that mass is conserved at each iteration. By introducing a preconditioner based on the weighted norms from the analysis, the performance of the iterative scheme is shown to be robust with respect to material and discretization parameters. By construction, the scheme is applicable to a wide range of locally conservative discretization schemes and we consider explicit examples in the framework of Mixed Finite Element methods. Finally, the theoretical results are confirmed with the use of numerical experiments.
\end{abstract}
%
%

\maketitle

\section{Introduction} 
\label{sec:introduction}

The coupled interaction between free flow and porous medium flow forms an active area of research due to their appearance in a wide variety of applications. Examples include biomedical applications such as blood filtration, engineering situations such as air filters and PEM fuel cells, as well as environmental considerations such as the drying of soils. In all mentioned applications, it is essential to properly capture the mutual interaction between a free, possibly turbulent, flow and the creeping flow inside the porous medium. 

We consider the simplified setting in which the free-flow regime is described by Stokes flow and we let Darcy's law govern the porous medium flow. Moreover, we only consider the case of stationary, single-phase flow and assume a sharp interface between the two flow regimes. These considerations are simplifications of a more general framework of models coupling free flow with porous medium flow. Such models have been the topic of a variety of scientific work in recent years, with focuses including mathematical analysis, discretization methods, and iterative solution techniques. Different formulations of the Stokes-Darcy problem have been analyzed in \cite{discacciati2004domain,layton2002coupling,gatica2011analysis}. Examples in the context of discretization methods include the use of Finite Volume methods \cite{iliev2004a, mosthaf2011a, rybak2015a, masson2016a, fetzer2017a, schneider2020coupling} and (Mixed) Finite Element methods, both in a coupled \cite{layton2002coupling,discacciati2009navier,riviere2005locally,gatica2009conforming} and in a unified \cite{armentano2019unified,karper2009unified} setting. Moreover, iterative methods for this problem are considered in e.g. \cite{discacciati2007robin,discacciati2005iterative,discacciati2018optimized,ganderderivation,galvis2007balancing,cao2011robin}. We refer the reader to the works \cite{discacciati2009navier,rybak2016mathematical,discacciati2004domain} and references therein for more comprehensive overviews on the results concerning the Stokes-Darcy model.

In order to distinguish this work from existing results, we formulate the following objective: \\
The goal of this work is to create an iterative numerical method that solves the stationary, Stokes-Darcy problem with the following three properties:
\begin{enumerate}
	\item \label{goal: mass conservation}
	The solution retains \textbf{local mass conservation}, after each iteration. 
	Since mass balance is a physical conservation law, we emphasize its importance over all other constitutive relationships in the model. Hence, the first aim is to produce a solution that respects local mass conservation and use iterations to improve the accuracy of the solution with respect to the remaining equations. Importantly, we aim to obtain a conservative flow field in the case that the iterative scheme is terminated before reaching convergence.

	We present two main ideas to achieve this. First, we limit ourselves to discretization methods capable of ensuring local mass conservation within each flow regime. Secondly, we ensure that no mass is lost across the Stokes-Darcy interface by introducing a single variable describing this interfacial flux. Our contribution in this context is to pose and analyze the Stokes-Darcy problem using function spaces that ensure normal flux continuity (Section~\ref{sub:functional_setting}), both in the continuous (Section~\ref{sec:well_posedness}) and discretized (Section~\ref{sec:discretization}) settings. Our approach is closely related to the ``global'' approach suggested in \cite[Remark 2.3.2]{discacciati2004domain} which we further develop in a functional setting. We moreover note that our construction is, in a sense, dual to the more conventional approach in which a mortar variable representing the interface pressure is introduced, see e.g. \cite{layton2002coupling,gatica2009conforming}.
	
	\item
	The performance of the iterative solution scheme is \textbf{robust with respect to physical and mesh parameters}. In this respect, the first aim is to obtain sufficient accuracy of the solution within a given number of iterations that is robust with respect to given material parameters such as the permeability of the porous medium and the viscosity of the fluid. This will allow the scheme to handle wide ranges of material parameters that arise either from the physical problem or due to the chosen non-dimensionalization.

	Robustness with respect to mesh size is advantageous from a computational perspective. If the scheme reaches sufficient accuracy within a number of iterations on a coarse grid, then this robustness provides a prediction on the necessary computational time on refined grids. We note that the analysis in this work is restricted to shape-regular meshes, hence the typical mesh size $h$ becomes the only relevant mesh parameter.

	To attain this goal, we pay special attention to the influence of the material and mesh parameters in the a priori analysis of the problem. We derive stability bounds of the solution in terms of functional norms weighted with the material parameters. One of the main contributions is thus the derivation of a properly weighted norm for the normal flux on the Stokes-Darcy interface, presented in equation \eqref{eq: norm phi}. In turn, this norm is used to construct an optimal preconditioner a priori.
	
	\item
	The method is easily extendable to a \textbf{wide range of discretization methods} for the Stokes and Darcy subproblems. Aside from compliance with aim (1), we impose as few restrictions as possible on the underlying choice of discretization methods, thus allowing the presented iterative scheme to be highly adaptable. Moreover, the scheme is able to benefit from existing numerical implementations that are tailored to solving the Stokes and Darcy subproblems efficiently. This work employs a conforming Mixed Finite Element method, keeping in mind that extensions can readily be made to other locally conservative methods such as e.g. Finite Volume Methods or Discontinuous Galerkin methods.

	In order to achieve this third goal, we first derive the properties of the problem in the continuous setting and apply the discretization afterward. The key strategy here is to reformulate the problem into a Steklov-Poincar\'e system concerning only the normal flux across the interface, similar to the strategy presented in \cite[Sec. 2.5]{discacciati2004domain}. We then propose a preconditioner for this problem that is independent of the chosen discretization methods for the subproblems.
\end{enumerate}

Our formulation and analysis of the Stokes-Darcy problem therefore has three distinguishing properties. Most importantly, we consider a mixed formulation of the coupled problem using a function space that strongly imposes normal flux continuity at the interface. In contrast, existing approaches often use a primal formulation for the Darcy subproblem \cite{discacciati2004domain} or enforce flux continuity using Lagrange multipliers \cite{layton2002coupling}. In the context of Mixed Finite Element Methods, this directly leads to different choices of discrete spaces. Secondly, our analysis employs weighted norms and we derive an estimate for the interface flux that has, to our knowledge, not been exploited in existing literature. Third, we propose a preconditioner in Section~\ref{sub:parameter_robust_preconditioning} that is entirely local to the interface and does not require additional subproblem solves, in contrast to more conventional approaches such as the Neumann-Neumann method presented in Section~\ref{sub:comparison_to_NN_method}. The construction of this preconditioner does, however, require solving a generalized eigenvalue problem, which is done in the a priori, or ``off-line'', stage. As an additional feature, our set-up does not require choosing any acceleration parameters.

The article is structured as follows. Section~\ref{sec:the_model} introduces the coupled Stokes-Darcy model and its variational formulation as well as the notational conventions used throughout this work. Well-posedness of the model is shown in Section~\ref{sec:well_posedness} with the use of weighted norms. Section~\ref{sec:the_steklov_poincare_system} shows the reduction to an interface problem concerning only the normal flux defined there. A conforming discretization is proposed in Section~\ref{sec:discretization} with the use of the Mixed Finite Element method. Using the ingredients of these sections, Section~\ref{sec:iterative_solvers} describes the proposed iterative scheme and the optimal preconditioner it relies on. The theoretical results are confirmed numerically in Section~\ref{sec:numerical_results}. Finally, Section~\ref{sec:conclusions} contains concluding remarks.

\section{The Coupled Stokes-Darcy Model} 
\label{sec:the_model}

Consider an open, bounded domain $\Omega \subset \mathbb{R}^n$, $n \in  \{2, 3\}$, decomposed into two disjoint, Lipschitz subdomains $\Omega_S$ and $\Omega_D$. Here, and throughout this work, the subscript $S$ or $D$ is used on subdomains and variables to denote its association to the Stokes or Darcy subproblem, respectively. Let the interface be denoted by $\Gamma := \partial{\Omega}_S \cap \partial{\Omega}_D$ and let $\bm{n}$ denote the unit vector normal to $\Gamma$ oriented outward with respect to $\Omega_S$. An illustration of these definitions is given in Figure~\ref{fig:figure1}.

\begin{figure}[ht]
	\centering
	\includegraphics[width = \textwidth]{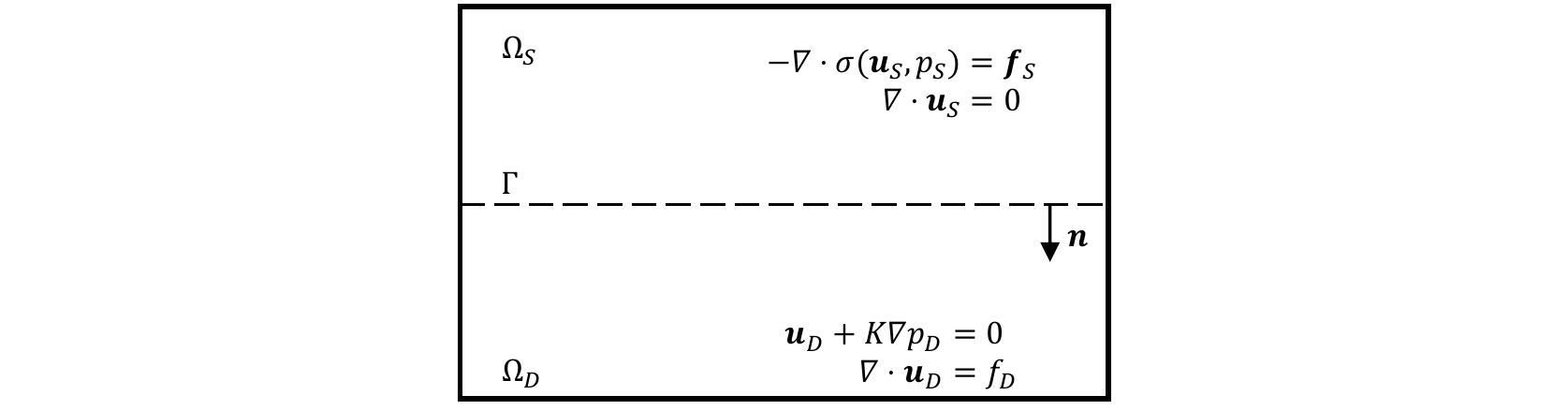}
	\caption{Decomposition of the domain into $\Omega_S$ and $\Omega_D$.}
	\label{fig:figure1}
\end{figure}

We introduce the model problem following the description of \cite{layton2002coupling}. The main variables are given by the velocity $\bm{u}$ and pressure $p$. A subscript denotes the restriction of a variable to the corresponding subdomain. The model is formed by considering Stokes flow for $(\bm{u}_S, p_S)$ in $\Omega_S$, Darcy flow for $(\bm{u}_D, p_D)$ in $\Omega_D$, and mass conservation laws for $\bm{u}$ in both subdomains: 
\begin{subequations} \label{eq: SD strong form}
	\begin{align}
		-\nabla \cdot \sigma(\bm{u}_S, p_S) &= \bm{f}_S, & 
		\nabla \cdot \bm{u}_S &= 0, & \text{in }\Omega_S, \\
		\bm{u}_D + K \nabla p_D &= 0, & 
		\nabla \cdot \bm{u}_D &= f_D, & \text{in }\Omega_D.
	\end{align}
In this setting, $K$ is the hydraulic conductivity of the porous medium. For simplicity, we assume that $K$ is homogeneous and isotropic and thus given by a positive scalar. On the right-hand side, $\bm{f}_S$ represents a body force and $f_D$ corresponds to a mass source.
In the governing equations for Stokes flow, let the strain $\varepsilon$ and stress $\sigma$ be given by
\begin{align*}
	\varepsilon(\bm{u}_S) &:= \frac{1}{2}\left(\nabla \bm{u}_S + (\nabla \bm{u}_S)^T \right), &
	\sigma(\bm{u}_S, p_S) &:= \mu \varepsilon(\bm{u}_S) - p_S I.
\end{align*}
The parameter $\mu > 0$ is the viscosity. 

Next, we introduce two coupling conditions on the interface $\Gamma$ that describe mass conservation and the balance of forces, respectively:
	\begin{align}
		\bm{n} \cdot \bm{u}_S &= \bm{n} \cdot \bm{u}_D, 
		& \text{on } \Gamma, \label{eq: coupling_mass}\\
		\bm{n} \cdot \sigma(\bm{u}_S, p_S) \cdot \bm{n} &= -p_D,
		& \text{on } \Gamma. 
	\end{align}
As remarked in the introduction, we keep a particular focus on conservation of mass. To ensure that no mass is lost across the interface, we will prioritize condition \eqref{eq: coupling_mass} at a later stage.

To close the model, we consider the following boundary conditions. First, for the Stokes subproblem, we impose the Beavers-Joseph-Saffman condition on the interface $\Gamma$, given by
\begin{align}
		\bm{n} \cdot \sigma(\bm{u}_S, p_S) \cdot \bm{\tau}
		&= - \beta \bm{\tau} \cdot \bm{u}_S, 
		& \text{on } \Gamma. 
		\label{eq: BJS}
\end{align}
Here, we define $\beta := \alpha \frac{\mu}{\sqrt{\bm{\tau} \cdot \kappa \cdot \bm{\tau}}} $ with $\kappa := \mu K$ the permeability and $\alpha$ a proportionality constant to be determined experimentally. Moreover, the unit vector $\bm{\tau}$ is obtained from the tangent bundle of $\Gamma$. Thus, for $n = 2$, equation \eqref{eq: BJS} corresponds to a single condition on the one-dimensional interface $\Gamma$ whereas for $n = 3$, it describes two separate coupling conditions.

The boundary of $\Omega$ is decomposed in the disjoint unions $\partial \Omega_S \setminus \Gamma = \partial_u \Omega_S \cup \partial_\sigma \Omega_S$ and $\partial \Omega_D \setminus \Gamma = \partial_u \Omega_D \cup \partial_p \Omega_D$. The subscript denotes the type of boundary condition imposed on that portion of the boundary. Specifically, we set
	\begin{align}
		\bm{u}_S &= 0, & \text{on } &\partial_u \Omega_S, & 
		\bm{n} \cdot \bm{u}_D &= 0, & \text{on } &\partial_u \Omega_D. \label{eq: BC essential} \\ 
		\bm{n} \cdot \sigma(\bm{u}_S, p_S) &= 0, & \text{on } &\partial_\sigma \Omega_S, &
		p_D &= g_p, & \text{on } &\partial_p \Omega_D, \label{eq: BC natural} 
	\end{align}
with $g_p$ a given pressure distribution.
\end{subequations}

In the following, we assume that the interface $\Gamma$ touches the portion of the boundary $\partial \Omega_S$ where homogeneous flux conditions are imposed, i.e. $\partial \Gamma \subseteq \overline{\partial_u \Omega_S}$. 
We note that this assumption excludes the case in which $\Omega_D$ is completely surrounded by $\Omega_S$.
Moreover, we assume that $|\partial_\sigma \Omega_S \cup \partial_p \Omega_D| > 0$ to ensure unique solvability of the coupled problem and we focus on the case in which $|\partial_\sigma \Omega_S| > 0$.

\subsection{Functional Setting} 
\label{sub:functional_setting}

In this section, we introduce the function spaces in which we search for a weak solution of problem \eqref{eq: SD strong form}. We start by considering the space for the velocity variable $\bm{u}$. With the aim of deriving mixed formulations for both subproblems, we introduce the following spaces:
\begin{subequations}
	\begin{align}
		\bm{V}_S &:= \left\{ \bm{v}_S \in (H^1(\Omega_S))^n :\ 
		\bm{v}_S|_{\partial_u \Omega_S} = 0 \right\}, \\
		\bm{V}_D &:= \left\{ \bm{v}_D \in H(\div, \Omega_D) :\ 
		\bm{n} \cdot \bm{v}_D|_{\partial_u \Omega_D} = 0 \right\}.
	\end{align}
\end{subequations}

Note that these spaces incorporate the boundary conditions \eqref{eq: BC essential} on $\partial_u \Omega$ which become essential boundary conditions in our mixed formulation. Similarly, the normal flux continuity across $\Gamma$ \eqref{eq: coupling_mass} needs to be incorporated as an essential boundary condition. For that, we introduce a single function $\phi \in \Lambda$, defined on $\Gamma$ to represent the normal flux across the interface. The next step is then to define the following three function spaces:
\begin{subequations}
	\begin{align}
		\bm{V}_S^0 &:= \left\{ \bm{v}_S \in \bm{V}_S :\ 
		\bm{n} \cdot \bm{v}_S|_{\Gamma} = 0 \right\}, \\
		\Lambda &:= H^{1/2}_{00}(\Gamma), \\
		\bm{V}_D^0 &:= \left\{ \bm{v}_D \in \bm{V}_D :\ 
		\bm{n} \cdot \bm{v}_D|_{\Gamma} = 0 \right\}.
	\end{align}
\end{subequations}
We note that $\Lambda$ is the normal trace space of $\bm{V}_S$ on $\Gamma$. From the previous section, we recall that $\Gamma$ touches the boundary $\partial \Omega$ where zero velocity conditions are imposed for the Stokes problem. The trace space is therefore characterized as the fractional Sobolev space $H^{1/2}_{00}(\Gamma)$, containing distributions that can be continuously extended by zero on $\partial \Omega$. We refer the reader to \cite{lions2012non} for more details on this type of trace spaces. For the purpose of our analysis, we note that the inclusion $H_0^1(\Gamma) \subset \Lambda \subset L^2(\Gamma)$ holds and we let $H^{-\half}(\Gamma)$ denote the dual of $\Lambda$.

For the incorporation of the interface condition \eqref{eq: coupling_mass} in our weak formulation, we introduce continuous operators that extend a given flux distribution on the interface to the two subdomains. The extension operators $\myR_S: \Lambda \to \bm{V}_S$ and $\myR_D: \Lambda \to \bm{V}_D$ are chosen such that
\begin{align} \label{eq: extension property}
	(\bm{n} \cdot \myR_S \varphi)|_{\Gamma} 
	&= 
	(\bm{n} \cdot \myR_D \varphi)|_{\Gamma} 
	= \varphi.
\end{align}
We use $\| \cdot \|_{s, \Omega}$ as short-hand notation for the norm on $H^s(\Omega)$. With this notation, the continuity of $\myR_i$ implies that the following inequalities hold
\begin{align} \label{eq: continuity R_i}
	\| \myR_S \varphi \|_{1, \Omega_S} &\lesssim \| \varphi \|_{\half, \Gamma}, &
	\| \myR_D \varphi \|_{0, \Omega_D} + \| \nabla \cdot \myR_D \varphi \|_{0, \Omega_D} &\lesssim \| \varphi \|_{-\half, \Gamma}.
\end{align}
Examples of continuous extension operators can be found in \cite[Sec. 4.1.2]{quarteroni1999domain}. The notation $A \lesssim B$ implies that a constant $c > 0$ exists, independent of material parameters and the mesh size $h$ such that $A \le cB$. The relationship $\gtrsim$ is defined analogously.

These definitions allow us to create a function space $\bm{V}$ containing velocities with normal trace continuity on $\Gamma$. Let this function space be defined as
\begin{align}
	\bm{V} &:= \left\{ \bm{v} \in (L^2(\Omega))^n :\ 
	\exists (\bm{v}_S^0, \varphi, \bm{v}_D^0) \in \bm{V}_S^0 \times \Lambda \times \bm{V}_D^0
	\text{ such that } \bm{v}|_{\Omega_i} = \bm{v}_i^0 + \myR_i \varphi, \text{ for } i \in \{S, D\} \right\}.
\end{align}

Second, the function space for the pressure variable is given by $W := L^2(\Omega)$ and we define $W_S := L^2(\Omega_S)$ and $W_D := L^2(\Omega_D)$.

As before, we use the subscript $i \in \{S, D\}$ to denote the restriction to a subdomain $\Omega_i$. Thus, for $(\bm{v}, w) \in \bm{V} \times W$, we have
\begin{align}
	\bm{v}_i &:= \bm{v}|_{\Omega_i} \in \bm{V}_i, & 
	w_i &:= w|_{\Omega_i} \in W_i, &
	i &\in \{S, D\}.
\end{align}
Despite the fact that each function in $\bm{V}$ can be decomposed into components in $\bm{V}_S$ and $\bm{V}_D$, we emphasize that $\bm{V}$ is a strict subspace of $\bm{V}_S \times \bm{V}_D$ due to the continuity of normal traces on $\Gamma$. 

A key concept in our functional setting is to consider a decomposition of $\bm{V}$ comprising of a function with zero normal trace on $\Gamma$ and an extension of the normal flux distribution. For that purpose, let $\bm{V}^0$ be the subspace of $\bm{V}$ consisting of functions with zero normal trace over $\Gamma$:
\begin{align}
	\bm{V}^0 := \left\{ \bm{v}^0 \in \bm{V} :\ 
		\exists (\bm{v}_S^0, \bm{v}_D^0) \in \bm{V}_S^0 \times \bm{V}_D^0
		\text{ such that } \bm{v}^0|_{\Omega_i} = \bm{v}_i^0, \text{ for } i \in \{S, D\} \right\}.
\end{align}

Secondly, we define the composite extension operator $\myR: \Lambda \to V$ such that $\myR \varphi|_{\Omega_i} = \myR_i \varphi$ for $i \in \{S, D\}$. Combined with the subspace $\bm{V}^0$, we obtain the decomposition
\begin{align} \label{eq: decomposition V}
	\bm{V} = \bm{V}^0 \oplus \myR \Lambda.
\end{align}

It is important to emphasize that the function space $\bm{V}$ is independent of the choice of extension operators. On the other hand, each choice of $\myR$ leads to a specific decomposition of the form \eqref{eq: decomposition V}.

\subsection{Variational Formulation} 
\label{sub:variational_Formulation}

With the function spaces defined, we continue by deriving the variational formulation of \eqref{eq: SD strong form}. The first step is to consider the Stokes and Darcy flow equations. We test these with $\bm{v} \in \bm{V}$ and integrate over the corresponding subdomain. Using $(\cdot, \cdot)_{\Omega}$ to denote the $L^2$ inner product on $\Omega$, we apply integration by parts and use the boundary conditions to derive
\begin{align*}
	-(\nabla \cdot \sigma(\bm{u}_S, p_S), \bm{v}_S)_{\Omega_S} &= \nonumber \\
	(\sigma(\bm{u}_S, p_S), \nabla \bm{v}_S)_{\Omega_S} 
	- ( \bm{n} \cdot \sigma(\bm{u}_S, p_S), \bm{v}_S )_{\Gamma} 
	&= \nonumber\\
	(\mu \varepsilon(\bm{u}_S), \varepsilon(\bm{v}_S))_{\Omega_S} 
	- (p_SI, \nabla \bm{v}_S)_{\Omega_S} 
	+ (\beta \bm{\tau} \cdot \bm{u}_S, \bm{\tau} \cdot \bm{v}_S )_{\Gamma}
	+ ( p_D, \bm{n} \cdot \bm{v}_S )_{\Gamma} 
	&= ( \bm{f}_S, \bm{v}_S )_{\Omega}.
\end{align*}

On the other hand, we test Darcy's law in the porous medium $\Omega_D$ and use similar steps to obtain
\begin{align*}
	(K^{-1} \bm{u}_D, \bm{v}_D)_{\Omega_D} 
	- (p_D, \nabla \cdot \bm{v}_D)_{\Omega_D} 
	- ( p_D, \bm{n} \cdot \bm{v}_D )_{\Gamma} 
	- ( g_p, \bm{n} \cdot \bm{v}_D )_{\partial_p \Omega_D}
	&= 0.
\end{align*}

The normal trace continuity imposed in the space $\bm{V}$ gives us
$( p_D, \bm{n} \cdot \bm{v}_S )_{\Gamma} - ( p_D, \bm{n} \cdot \bm{v}_D )_{\Gamma} = 0$.

In turn, after supplementing the system with the equations for mass conservation, we arrive at the following variational formulation: \\
Find the pair $(\bm{u}, p) \in \bm{V} \times W$ that satisfies
\begin{subequations}
\begin{align}
	(\mu \varepsilon(\bm{u}_S), \varepsilon(\bm{v}_S))_{\Omega_S} 
	+ (\beta \bm{\tau} \cdot \bm{u}_S, \bm{\tau} \cdot \bm{v}_S )_{\Gamma} 
	+ (K^{-1} \bm{u}_D, \bm{v}_D)_{\Omega_D} 
	& \nonumber\\
	- (p_S, \nabla \cdot \bm{v}_S)_{\Omega_S} 
	- (p_D, \nabla \cdot \bm{v}_D)_{\Omega_D} 
	&= ( \bm{f}_S, \bm{v}_S )_{\Omega_S}
	+ ( g_p, \bm{n} \cdot \bm{v}_D )_{\partial_p \Omega_D}, 
	& \forall \bm{v} &\in \bm{V}, \\
	(\nabla \cdot \bm{u}_S, w_S)_{\Omega_S} 
	+ (\nabla \cdot \bm{u}_D, w_D)_{\Omega_D} 
	&= 
	(f_D, w_D)_{\Omega_D},
	& \forall w &\in W.
\end{align}
\end{subequations}

We note that this system has a characteristic saddle-point structure, allowing us to rewrite the problem as:\\
Find the pair $(\bm{u}, p) \in \bm{V} \times W$ that satisfies
\begin{subequations}  \label{eq: variational formulation}
	\begin{align} 
		a(\bm{u}, \bm{v}) + b(\bm{v}, p) &= f_u(\bm{v}),
		& \forall \bm{v} &\in \bm{V}, \label{eq: variational formulation 1st eq}\\
		b(\bm{u}, w) &= f_p(w), \label{eq: variational formulation 2nd eq}
		& \forall w &\in W.
	\end{align}
\end{subequations}

The bilinear forms $a: \bm{V} \times \bm{V} \to \mathbb{R}$ and $b: \bm{V} \times W \to \mathbb{R}$, and the functionals $f_u: \bm{V} \to \mathbb{R}$ and $f_p: W \to \mathbb{R}$ are given by
\begin{subequations} \label{eq: bilinear forms}
	\begin{align}
		a(\bm{u}, \bm{v}) &:= (\mu \varepsilon(\bm{u}_S), \varepsilon(\bm{v}_S))_{\Omega_S} 
		+ (\beta \bm{\tau} \cdot \bm{u}_S, \bm{\tau} \cdot \bm{v}_S )_{\Gamma} 
		+ (K^{-1} \bm{u}_D, \bm{v}_D)_{\Omega_D}, \\
		b(\bm{u}, w) &:= -(\nabla \cdot \bm{u}_S, w_S)_{\Omega_S}
		-(\nabla \cdot \bm{u}_D, w_D)_{\Omega_D}, \\
		f_u(\bm{v}) &:= ( \bm{f}_S, \bm{v}_S )_{\Omega_S} 
		+ ( g_p, \bm{n} \cdot \bm{v}_D )_{\partial_p \Omega_D}, \\
		f_p(w) &:= 
		-(f_D, w_D)_{\Omega_D}.
	\end{align}
\end{subequations}


\section{Well-Posedness Analysis} 
\label{sec:well_posedness}

In this section, we analyze problem \eqref{eq: variational formulation} with the use of weighted norms. The main goal is to show that a unique solution exists that is bounded in norms that depend on the material parameters. Consequently, this result allows us to construct an iterative method that is robust with respect to material parameters. 
We start by deriving the appropriate norms, for which we first make two assumptions on the material parameters. 
	

First, the constant $\beta$ in the Beavers-Joseph-Saffman condition \eqref{eq: BJS} is assumed to be bounded as
\begin{subequations} \label{eqs: material parameter bounds}
\begin{align}
	\beta = \alpha \frac{\mu}{\sqrt{\bm{\tau} \cdot \kappa \cdot \bm{\tau}}} \lesssim \mu.
\end{align}
In the special case of $\alpha = 0$, this condition is trivially satisfied.

Second, we assume that the permeability $\kappa := \mu K$ is bounded from above in the sense that
\begin{align}
	\mu K \lesssim 1.
\end{align}
\end{subequations}

We are now ready to define the weighted norms for $\bm{v} \in \bm{V}$ and $w \in W$, respectively, given by
\begin{subequations} \label{eq: norms}
	\begin{align}
		\| \bm{v} \|_V^2 &:= 
		\| \mu^{\half} \bm{v}_S \|_{1, \Omega_S}^2
		+ \| K^{-\half} \bm{v}_D \|_{0, \Omega_D}^2
		+ \| K^{-\half} \nabla \cdot \bm{v}_D \|_{0, \Omega_D}^2 \\
		\| w \|_W^2 &:= \| \mu^{-\half} w_S \|_{0, \Omega_S}^2
		+ \| K^{\half} w_D \|_{0, \Omega_D}^2.
	\end{align}
\end{subequations}

The next step is to analyze the problem using these norms. For that purpose, we recall the identification of \eqref{eq: variational formulation} as a saddle-point problem. Using saddle-point theory \cite{boffi2013mixed}, well-posedness is shown by proving the four sufficient conditions presented in the following lemma.
\begin{lemma} \label{lem: inequalities}
	The bilinear forms defined in \eqref{eq: bilinear forms} satisfy the following inequalities:
	\begin{subequations}
	\begin{align}
			& \text{For } \bm{u}, \bm{v} \in \bm{V}: 
			&  a(\bm{u}, \bm{v}) &\lesssim \| \bm{u} \|_V \| \bm{v} \|_V. \label{ineq: a_cont}\\
			& \text{For } (\bm{v}, w) \in \bm{V} \times W: 
			&  b(\bm{v}, w) &\lesssim \| \bm{v} \|_V \| w \|_W. \label{ineq: b_cont}\\
			& \text{For } \bm{v} \in \bm{V} \text{ with } b(\bm{v}, w) = 0 \ \forall w \in W: 
			&  a(\bm{v}, \bm{v}) &\gtrsim \| \bm{v} \|_V^2. \label{ineq: a_coercive}\\
			& \text{For } w \in W, \ \exists \bm{v} \in \bm{V} \text{ with } \bm{v} \ne 0, \text{ such that}: 
			&  b(\bm{v}, w) &\gtrsim \| \bm{v} \|_V \| w \|_W.\label{ineq: b_infsup}
		\end{align}
	\end{subequations}
	\end{lemma}
\begin{proof}
	Using the Cauchy-Schwarz inequality, the assumptions \eqref{eqs: material parameter bounds}, and a trace inequality for $H^1$, we obtain the continuity bounds \eqref{ineq: a_cont} and \eqref{ineq: b_cont}:	
\begin{subequations} \label{ineqs: continuity}
	\begin{align}
		a(\bm{u}, \bm{v}) 
		&\lesssim \|\mu^{\half} \bm{u}_S \|_{1, \Omega_S} 
		\|\mu^{\half} \bm{v}_S \|_{1, \Omega_S} 
		+ \| \beta^{\half} \bm{\tau} \cdot \bm{u}_S \|_{0, \Gamma}
		\| \beta^{\half} \bm{\tau} \cdot \bm{v}_S \|_{0, \Gamma} 
		+ \| K^{-\half} \bm{u}_D \|_{0, \Omega_D}
		\| K^{-\half} \bm{v}_D \|_{0, \Omega_D} \nonumber\\
		&\lesssim \| \bm{u} \|_V \| \bm{v} \|_V, \\
		b(\bm{v}, w) 
		&\lesssim \| \mu^{\half} \bm{v}_S \|_{1, \Omega_S} \| \mu^{-\half} w_S \|_{0, \Omega_S} 
		+ \| K^{-\half} \nabla \cdot \bm{v}_D \|_{0, \Omega_D} \| K^{\half} w_D \|_{0, \Omega_D} \nonumber\\
		&\lesssim \| \bm{v} \|_V \| w \|_W.
	\end{align}
\end{subequations}

	For the proof of inequality \eqref{ineq: a_coercive}, we first note that if $b(\bm{v}, w) = 0$ for all $w \in W$, then $\nabla \cdot \bm{v}_D = 0$. Combining this observation with Korn's inequality gives us:
	\begin{align} \label{eq: proof 3.5c}
		a(\bm{v}, \bm{v}) &\gtrsim
		\| \mu^{\half} \bm{v}_S \|_{1, \Omega_S}^2
		+ \| \beta^{\half} \bm{\tau} \cdot \bm{v}_S \|_{0, \Gamma}^2
		+ \| K^{-\half} \bm{v}_D \|_{0, \Omega_D}^2 \nonumber\\
		&\ge
		\| \mu^{\half} \bm{v}_S \|_{1, \Omega_S}^2
		+ \| K^{-\half} \nabla \cdot \bm{v}_D \|_{0, \Omega_D}^2 
		+ \| K^{-\half} \bm{v}_D \|_{0, \Omega_D}^2
		= \| \bm{v} \|_V^2.
	\end{align}

	Inequality \eqref{ineq: b_infsup} is the inf-sup condition relevant for this formulation. For a given $w = (w_S, w_D) \in W$, let us construct $\bm{v} = (\bm{v}_S^0, \phi, \bm{v}_D^0) \in \bm{V}$ in the following manner. First, let the interface function $\phi \in H_0^1(\Gamma)$ solve the following, constrained minimization problem:
	\begin{align} \label{eq: phi constraint}
		\min_{\varphi \in H_0^1(\Gamma)} \tfrac12 &\| \varphi \|_{1, \Gamma}^2,
		&\text{subject to } (\varphi, 1)_{\Gamma} &= (K w_D, 1)_{\Omega_D}.
	\end{align}
	The solution $\phi$ then satisfies the two key properties:
	\begin{subequations}
	\begin{align}
		\| \phi \|_{1, \Gamma} &\lesssim \| K w_D \|_{0, \Omega_D}, \label{eq: bound on phi} \\
		(\nabla \cdot \myR_D \phi, 1)_{\Omega_D} &=
		(- \bm{n} \cdot \myR_D \phi, 1)_{\Gamma} = 
		(- \phi, 1)_{\Gamma} = 
		- (K w_D, 1)_{\Omega_D}. \label{eq: compatibility of phi}
	\end{align}
	\end{subequations}
	Bound \eqref{eq: bound on phi} can be deduced by constructing a function $\psi \in H_0^1(\Gamma)$ that satisfies the constraint in \eqref{eq: phi constraint} and is bounded in the sense of \eqref{eq: bound on phi}. It then follows that the minimizer $\phi$ satisfies \eqref{eq: bound on phi} as well.
	
	Next, we construct $\bm{v}_i^0 \in \bm{V}_i^0$ for $i \in \{S, D\}$. For that, we first introduce $p_S \in H^2(\Omega_S)$ as the solution to the following auxiliary problem
	\begin{subequations} \label{eqs: aux prob p_S}
	\begin{align}
		- \nabla \cdot \nabla p_S &= \mu^{-1} w_S + \nabla \cdot \myR_S \phi, \\
		p_S|_{\partial_\sigma \Omega_S} &= 0, \\
		(\bm{n} \cdot \nabla p_S)|_{\partial \Omega_S \setminus \partial_\sigma \Omega_S} &= 0.
	\end{align}
	\end{subequations}
	Similarly, we define $p_D \in H^2(\Omega_D)$ such that
	\begin{subequations} \label{eqs: aux prob p_D}
	\begin{align}
		- \nabla \cdot \nabla p_D &= K w_D + \nabla \cdot \myR_D \phi, \\
		(\bm{n} \cdot \nabla p_D)|_{\partial \Omega_D} &= 0.
	\end{align}
	\end{subequations}
	We note that \eqref{eqs: aux prob p_D} is a Neumann problem. We therefore verify the compatibility of the right hand side by using \eqref{eq: compatibility of phi} in the following derivation:
	\begin{align*}
		( K w_D + \nabla \cdot \myR_D \phi, 1)_{\Omega_D}
		= ( K w_D - K w_D, 1)_{\Omega_D}
		= 0.
	\end{align*}
	
	Let $\bm{v}_S^0 := \nabla p_S$ and $\bm{v}_D^0:= \nabla p_D$. From the elliptic regularity of the auxiliary problems, see e.g. \cite{evans2010partial}, we obtain the bounds
	\begin{subequations}
	\begin{align}
		\| \bm{v}_S^0 \|_{1, \Omega_S} \lesssim 
		\| \mu^{-1} w_S \|_{0, \Omega_S}  + \| \nabla \cdot \myR_S \phi \|_{0, \Omega_S}, \\
		\| \bm{v}_D^0 \|_{1, \Omega_D} \lesssim 
		\| K w_D \|_{0, \Omega_D}  + \| \nabla \cdot \myR_D \phi \|_{0, \Omega_D}.
	\end{align}
	\end{subequations}

	Next, we set $\bm{v}_S := \bm{v}_S^0 + \myR_S \phi$ and $\bm{v}_D := \bm{v}_D^0 + \myR_D \phi$. Combining the bounds on $\bm{v}_S^0$ and $\phi$ with the continuity of the extension operators from \eqref{eq: continuity R_i} and the material parameter bounds \eqref{eqs: material parameter bounds}, we derive
\begin{subequations}
	\begin{align}
		\| \mu^{\half} \bm{v}_S \|_{1, \Omega_S}
		&\le \| \mu^{\half} \bm{v}_S^0 \|_{1, \Omega_S} 
		+ \| \mu^{\half} \myR_S \phi \|_{1, \Omega_S} \nonumber \\
		&\lesssim \| \mu^{-\half} w_S \|_{0, \Omega_S} 
		+ \| \mu^{\half} \nabla \cdot \myR_S \phi \|_{0, \Omega_S} 
		+ \| \mu^{\half} \myR_S \phi \|_{1, \Omega_S} \nonumber \\
		&\lesssim \| \mu^{-\half} w_S \|_{0, \Omega_S} 
		+ \| \mu^{\half} \phi \|_{\half, \Gamma} \nonumber \\
		&\lesssim \| \mu^{-\half} w_S \|_{0, \Omega_S} 
		+ \| \mu^{\half} K w_D \|_{0, \Omega_D}  \nonumber \\
		&\lesssim \| \mu^{-\half} w_S \|_{0, \Omega_S} 
		+ \| K^{\half} w_D \|_{0, \Omega_D}.
	\end{align}
	Similarly, $\bm{v}_D$ is bounded in the following sense:
	\begin{align}
		\| K^{-\half} \bm{v}_D \|_{0, \Omega_D}
		+ \| K^{-\half} \nabla \cdot \bm{v}_D \|_{0, \Omega_D}
		&\le 
		\| K^{-\half} \bm{v}_D^0 \|_{0, \Omega_D}
		+ \| K^{-\half} \myR_D \phi \|_{0, \Omega_D}
		+ \| K^{\half} w_D \|_{0, \Omega_D} \nonumber \\
		&\lesssim
		\| K^{-\half} \nabla \cdot \myR_D \phi \|_{0, \Omega_D}
		+ \| K^{-\half} \myR_D \phi \|_{0, \Omega_D}
		+ \| K^{\half} w_D \|_{0, \Omega_D} \nonumber \\
		&\lesssim
		\| K^{-\half} \phi \|_{-\half, \Gamma}
		+ \| K^{\half} w_D \|_{0, \Omega_D} \nonumber \\
		&\lesssim
		\| K^{\half} w_D \|_{0, \Omega_D}.
	\end{align}
	\end{subequations}
	In the final step, we have used that $H^1(\Gamma) \subseteq H^{-\half}(\Gamma)$ and \eqref{eq: bound on phi}.
	
	By construction, $\bm{v}$ now satisfies the following two properties:
	\begin{subequations} \label{eqs: proof b inf sup}
	\begin{align}
		\| \bm{v} \|_V 
		&= \left(\| \mu^{\half} \bm{v}_S \|_{1, \Omega_S}^2
				+ \| K^{-\half} \bm{v}_D \|_{0, \Omega_D}^2
				+ \| K^{-\half} \nabla \cdot \bm{v}_D \|_{0, \Omega_D}^2 \right)^\half
		\lesssim \| w \|_W, \\
		b(\bm{v}, w) &= -(\nabla \cdot \bm{v}_S, w_S)_{\Omega_S}
		-(\nabla \cdot \bm{v}_D, w_D)_{\Omega_D} \nonumber\\
		&= -(\nabla \cdot (\nabla p_S + \myR_S \phi), w_S)_{\Omega_S}
		-(\nabla \cdot (\nabla p_D + \myR_D \phi), w_D)_{\Omega_D}\nonumber\\
		&= \| \mu^{-\half} w_S \|_{0, \Omega_S}^2
		+ \| K^{\half} w_D \|_{0, \Omega_D}^2 \nonumber\\
		&=\| w \|_W^2.
	\end{align}
	\end{subequations}
	The proof is concluded by gathering \eqref{eqs: proof b inf sup}.
\end{proof}

In the special case of $|\partial_p \Omega_D| > 0$, the Darcy subproblem is itself well-posed. This can be used to our advantage in the proof of \eqref{ineq: b_infsup}. In particular, the construction of $\phi \in \Lambda$ becomes obsolete, as shown in the following corollary.

\begin{corollary} \label{cor: infsup V0}
	If $|\partial_p \Omega_D| > 0$, then for each $w \in W$, there exists $\bm{v}^0 \in \bm{V}^0$ with $\bm{v}^0 \ne 0$ such that 
	\begin{align*}
		b(\bm{v}^0, w) \gtrsim \| \bm{v}^0 \|_V \| w \|_W.
	\end{align*}
\end{corollary}
\begin{proof}
	We follow the same arguments as for \eqref{ineq: b_infsup} in Lemma~\ref{lem: inequalities}. The main difference is that we now set $\phi = 0$ and solve auxiliary Poisson problems to obtain $(\bm{v}_S^0, \bm{v}_D^0) \in \bm{V}_S^0 \times \bm{V}_D^0$ such that
	\begin{align*}
		- \nabla \cdot \bm{v}_S^0 &= \mu^{-1} w_S, &
		- \nabla \cdot \bm{v}_D^0 &= K w_D, \\
		(\bm{n} \cdot \bm{v}_S^0)|_{\partial \Omega_S \setminus \partial_\sigma \Omega_S} &= 0, &
		(\bm{n} \cdot \bm{v}_D^0)|_{\partial \Omega_D \setminus \partial_p \Omega_D} &= 0.
	\end{align*}
	Since both $\partial_\sigma \Omega_S$ and $\partial_p \Omega_D$ have positive measure, these two subproblems are well-posed and the statement follows by elliptic regularity.
\end{proof}

We are now ready to present the main result of this section, namely that problem \eqref{eq: variational formulation} is well-posed with respect to the weighted norms of \eqref{eq: norms}.
\begin{theorem} \label{thm: well-posedness}
	Problem \eqref{eq: variational formulation} is well-posed, i.e. a unique solution $(\bm{u}, p) \in \bm{V} \times W$ exists satisfying
	\begin{align}
		\| \bm{u} \|_V + \| p \|_W 
		\lesssim 
		\| \mu^{-\half} \bm{f}_S \|_{-1, \Omega_S}
		+ \| K^{-\half} f_D \|_{0, \Omega_D}
		+ \| K^{\half} g_p \|_{\half, \partial_p \Omega_D}.
	\end{align}
\end{theorem}
\begin{proof}
	With the inequalities from \Cref{lem: inequalities}, it suffices to show continuity of the right-hand side. Let us therefore apply the Cauchy-Schwarz inequality followed by a trace inequality:
	\begin{align*}
		f_u(\bm{v}) + f_p(w) &= ( \bm{f}_S, \bm{v}_S )_{\Omega_S} 
		+ (f_D, w_D)_{\Omega_D} 
		+ ( g_p, \bm{n} \cdot \bm{v}_D )_{\partial_p \Omega_D} \\
		&\le \| \mu^{-\half} \bm{f}_S \|_{-1, \Omega_S} \| \mu^{\half} \bm{v}_S \|_{1, \Omega_S}
		+ \| K^{-\half} f_D \|_{0, \Omega_D} \| K^{\half} w_D \|_{0, \Omega_D} 
		\\
		& \ \ + \| K^{\half} g_p \|_{\half, \partial_p \Omega_D} \| K^{-\half} \bm{n} \cdot \bm{v}_D \|_{-\half, \partial_p \Omega_D} \\
		&\lesssim \left( \| \mu^{-\half} \bm{f}_S \|_{-1, \Omega_S} 
		+ \| K^{-\half} f_D \|_{0, \Omega_D} 
		+ \| K^{\half} g_p \|_{\half, \partial_p \Omega_D} \right) 
		\left( \| \bm{v} \|_{V} + \| w \|_W \right).
	\end{align*}
	With the continuity of the right-hand shown, all requirements are satisfied to invoke standard saddle point theory \cite{boffi2013mixed}, proving the claim.
\end{proof}

\section{The Steklov-Poincar\'e System}
\label{sec:the_steklov_poincare_system}

The strategy is to introduce the Steklov-Poincar\'e operator $\Sigma$ and reduce the system \eqref{eq: variational formulation} to a problem concerning only the interface flux $\phi$. The reason for this is twofold. First, since the interface is a lower-dimensional manifold, the problem is reduced in dimensionality and is therefore expected to be easier to solve. Second, we show that the resulting system is symmetric and positive-definite and hence amenable to a large class of iterative solvers including the Minimal Residual (MinRes) and the Conjugate Gradient (CG) method.

We start with the case in which both the pressure and stress boundary conditions are prescribed on a part of the boundary with positive measure, i.e. we assume that $| \partial_\sigma \Omega_S | > 0$ and $| \partial_p \Omega_D | > 0$. The cases in which one, or both, of the subproblems have pure Neumann boundary conditions are considered afterward.

In order to construct the reduced problem, we use the bilinear forms and functionals from \eqref{eq: bilinear forms} and the extension operator $\myR$ from Section~\ref{sub:functional_setting} and define the operator $\Sigma: \Lambda \to \Lambda^*$ and $\chi \in \Lambda^*$ as
\begin{subequations} 
\begin{align}
	\langle \Sigma \phi, \varphi \rangle &:=
	a(\bm{u}_\star^0 + \myR \phi, \myR \varphi) + b(\myR \varphi, p_\star), \label{eq: def Sigma}\\
	\langle \chi, \varphi \rangle &:=
	f_u(\myR \varphi) - a(\bm{u}_0^0, \myR \varphi) - b(\myR \varphi, p_0),
\end{align}
\end{subequations}	
in which $\langle \cdot, \cdot \rangle$ denotes the duality pairing on $\Lambda^* \times \Lambda$. Here, the pair $(\bm{u}_\star^0, p_\star) \in \bm{V}^0 \times W$ satisfies
\begin{subequations} \label{eq: auxiliary problem _phi}
	\begin{align}
		a(\bm{u}_\star^0, \bm{v}^0) + b(\bm{v}^0, p_\star)
		&= - a(\myR \phi, \bm{v}^0), 
		& \forall \bm{v}^0 &\in \bm{V}^0, \\
		b(\bm{u}_\star^0, w) 
		&= - b(\myR \phi, w), 
		& \forall w &\in W.
		\label{eq aux problem eq2}
	\end{align}
\end{subequations}
and the pair $(\bm{u}_0^0, p_0) \in \bm{V}^0 \times W$ is defined such that
\begin{subequations} \label{eq: auxiliary problem _0}
\begin{align}
		a(\bm{u}_0^0, \bm{v}^0) + b(\bm{v}^0, p_0)
		&= f_u(\bm{v}^0),
		& \forall \bm{v}^0 &\in \bm{V}^0, \\
		b(\bm{u}_0^0, w) 
		&= f_p(w), 
		& \forall w &\in W.
\end{align}
\end{subequations}

With the above definitions, we introduce the reduced interface problem as: \\
Find $\phi \in \Lambda$ such that 
\begin{align} \label{eq: poincare steklov}
	\langle \Sigma \phi, \varphi \rangle &=
	\langle \chi, \varphi \rangle,
	& \forall \varphi &\in \Lambda.
\end{align}

Note that setting $\bm{u} := \bm{u}_\star^0 + \bm{u}_0^0 + \myR \phi$ and $p := p_\star + p_0$ yields the solution to the original problem \eqref{eq: variational formulation}. Hence, if this problem admits a unique solution, then \eqref{eq: poincare steklov} and \eqref{eq: variational formulation} are equivalent. 

Similar to the analysis of problem~\eqref{eq: variational formulation} in Section~\ref{sec:well_posedness}, we require an appropriate, parameter-dependent norm on functions $\varphi \in \Lambda$ in order to analyze \eqref{eq: poincare steklov}. Let us therefore define
\begin{align} \label{eq: norm phi}
		\| \varphi \|_{\Lambda}^2 &:= 
		\| \mu^{\half} \varphi \|_{\half, \Gamma}^2
		+ \| K^{-\half} \varphi \|_{-\half, \Gamma}^2.
\end{align}
We justify this choice by proving two bounds with respect to $\| \cdot \|_V$ in the following lemma. These results are then used in a subsequent theorem to show that $\Sigma$ is continuous and coercive with respect to $\| \cdot \|_\Lambda$.

\begin{lemma} \label{lem: norm equivalences}
	Given $\phi \in \Lambda$, then the following bounds hold for any $\bm{u}^0 \in \bm{V}^0$:
	\begin{align}
		\| \phi \|_{\Lambda}
		&\lesssim \| \bm{u}^0 + \myR \phi \|_V, &
		\| \myR \phi \|_V &\lesssim 
		\| \phi \|_{\Lambda}.
	\end{align}
\end{lemma}
\begin{proof}
	We apply trace inequalities in $H^1(\Omega_S)$ and $H(\div, \Omega_D)$:
	\begin{align*}
			\| \phi \|_{\Lambda}^2
			&= \| \mu^{\half} \phi \|_{\half, \Gamma}^2
			+ \| K^{-\half} \phi \|_{-\half, \Gamma}^2 \nonumber\\
			&\lesssim\| \mu^{\half} (\bm{u}_S^0 + \myR_S \phi) \|_{1, \Omega_S}^2
			+ \| K^{-\half} (\bm{u}_D^0 + \myR_D \phi) \|_{0, \Omega_D}^2
			+ \| K^{-\half} \nabla \cdot (\bm{u}_D^0 + \myR_D \phi) \|_{0, \Omega_D}^2
			= \| \bm{u}^0 + \myR \phi \|_V^2.
	\end{align*}
	Thus, the first inequality is shown. On the other hand, the continuity of $\myR_i$ for $i \in \{S, D\}$ from \eqref{eq: continuity R_i} gives us
	\begin{align*}
		\| \myR \phi \|_V^2
		&\le \| \mu^{\half} \myR_S \phi \|_{1, \Omega_S}^2
		+ \| K^{-\half} \myR_D \phi \|_{0, \Omega_D}^2
		+ \| K^{-\half} \nabla \cdot \myR_D \phi \|_{0, \Omega_D}^2 \nonumber\\
		&\lesssim \| \mu^{\half} \phi \|_{\half, \Gamma}^2
		+ \| K^{-\half} \phi \|_{-\half, \Gamma}^2
		= \| \phi \|_{\Lambda}^2.
	\end{align*}
\end{proof}

\begin{theorem} \label{thm: sigma SPD}
	The operator $\Sigma: \Lambda \to \Lambda^*$ is symmetric, continuous, and coercive with respect to the norm $\| \cdot \|_{\Lambda}$. 
\end{theorem}
\begin{proof}
	We first note that the auxiliary problem \eqref{eq: auxiliary problem _phi} is well-posed by Lemma~\ref{lem: inequalities}, Corollary~\ref{cor: infsup V0}, and saddle point theory. Moreover, the right-hand side is continuous due to \eqref{ineq: a_cont} and \eqref{ineq: b_cont}. For given $\phi$, the pair $(\bm{u}_\star^0, p_\star)$ therefore exists uniquely and satisfies
	\begin{align} \label{eq: bound u_phi}
		\| \bm{u}_\star^0 \|_V + \| p_\star \|_W 
		\lesssim 
		\| \myR \phi \|_V.
	\end{align}

	Symmetry is considered next. Let $(\bm{u}_\varphi, p_\varphi)$ be the solution to \eqref{eq: auxiliary problem _phi} with data $\varphi$. By setting $(\bm{v}^0, w) = (\bm{u}_\star^0, p_\star)$ in the corresponding problem, it follows that
	\begin{align*}
		a(\bm{u}_\varphi^0, \bm{u}_\star^0) 
		+ b(\bm{u}_\star^0, p_\varphi)
		+ b(\bm{u}_\varphi^0, p_\star) 
		&=
		- a(\myR \varphi, \bm{u}_\star^0) - b(\myR \varphi, p_\star) 
	\end{align*}
	Substituting this in definition \eqref{eq: def Sigma} and using the symmetry of $a$, we obtain
	\begin{align}
		\langle \Sigma \phi, \varphi \rangle
		&= a(\myR \phi, \myR \varphi) + a(\bm{u}_\star^0, \myR \varphi) + b(\myR \varphi, p_\star) 
		= a(\myR \phi, \myR \varphi) - a(\bm{u}_\star^0, \bm{u}_\varphi^0) - b(\bm{u}_\varphi^0, p_\star) 
		 - b(\bm{u}_\star^0, p_\varphi),
	\end{align}
	and symmetry of $\Sigma$ is shown.

	We continue by proving continuity of $\Sigma$. Employing \eqref{ineq: a_cont} and \eqref{ineq: b_cont} once again, it follows that
	\begin{align}
		\langle \Sigma \phi, \varphi \rangle
		\lesssim 
		(\| \bm{u}_\star^0 \|_V + \| \myR \phi \|_V + \| p_\star \|_W)
		\| \myR \varphi \|_V
		\lesssim 
		\| \myR \phi \|_V
		\| \myR \varphi \|_V
		\lesssim 
		\| \phi \|_\Lambda
		\| \varphi \|_\Lambda
	\end{align}
	in which the second and third inequalities follow from \eqref{eq: bound u_phi} and Lemma~\ref{lem: norm equivalences}, respectively.

	It remains to show coercivity, which we derive by setting $\varphi = \phi$ and $(\bm{v}^0, w) = (\bm{u}_\star^0, p_\star)$ in \eqref{eq: auxiliary problem _phi}:
	\begin{align}
		\langle \Sigma \phi, \phi \rangle
		&= a(\bm{u}_\star^0 + \myR \phi, \myR \phi) + b(\myR \phi, p_\star) 
		= a(\bm{u}_\star^0 + \myR \phi, \myR \phi) + a(\bm{u}_\star^0 + \myR \phi, \bm{u}_\star^0) 
		= a(\bm{u}_\star^0 + \myR \phi, \bm{u}_\star^0 + \myR \phi).
	\end{align}
	Next, we observe that \eqref{eq aux problem eq2} gives us $b(\bm{u}_\star^0 + \myR \phi, w) = 0$ for all $w \in W$. Thus, we use \eqref{ineq: a_coercive} and Lemma~\ref{lem: norm equivalences} to conclude that
	\begin{align}
		\langle \Sigma \phi, \phi \rangle 
		&\gtrsim \| \bm{u}_\star^0 + \myR \phi \|_V^2
		\gtrsim \| \phi \|_\Lambda^2.
	\end{align}
\end{proof}

\begin{corollary}
	Problem \eqref{eq: poincare steklov} is well-posed and the solution $\phi \in \Lambda$ satisfies
	\begin{align}
		\| \phi \|_{\Lambda}
		\lesssim 
		\| \mu^{-\half} \bm{f}_S \|_{-1, \Omega_S}
		+ \| K^{-\half} f_D \|_{0, \Omega_D}
		+ \| K^{\half} g_p \|_{\half, \partial_p \Omega_D}.
	\end{align}
\end{corollary}
\begin{proof} 
	As shown in Theorem~\ref{thm: sigma SPD}, $\Sigma$ is symmetric and positive-definite. Therefore, \eqref{eq: poincare steklov} admits a unique solution. We then set $\bm{u} = \bm{u}_\star^0 + \bm{u}_0^0 + \myR \phi$ and $p = p_\star + p_0$ and note that $(\bm{u}, p)$ is the solution to \eqref{eq: variational formulation}. By employing Lemma~\ref{lem: norm equivalences}, we note that $\| \phi \|_{\Lambda} \lesssim \| \bm{u} \|_V + \| p \|_W$ and the proof is concluded using the result from Theorem~\ref{thm: well-posedness}.
\end{proof}

\subsection{Neumann Problems}
\label{sub:neumann_cases}

In this section, we consider the case in which one, or both, of the subproblems have flux boundary conditions prescribed on the entire boundary. In other words, the cases in which $| \partial_\sigma \Omega_S | = 0$ or $| \partial_p \Omega_D | = 0$. 
We first introduce the setting in which one of the subdomains corresponds to a Neumann problem, followed by the case of $| \partial_\sigma \Omega_S | = | \partial_p \Omega_D | = 0$.

\subsubsection{Single Neumann Subproblem}
\label{ssub:single_neumann_problem}

Let us consider $| \partial_p \Omega_D | = 0$ and $| \partial_\sigma \Omega_S | > 0$, noting that the converse case follows by symmetry. The complication in this case is that solving the Darcy subproblem results in a pressure distribution that is defined up to a constant. Thus, several preparatory steps need to be made before the interface problem can be formulated and solved. 

The key idea is to pose the interface problem on the subspace of $\Lambda$ containing functions with zero mean. This is done by introducing a function $\phi_\star$ that balances the source term in $\Omega_D$ and subtracting this from the problem. The modified interface problem produces a pressure distribution with zero mean in $\Omega_D$ and we obtain the true pressure average obtained afterwards.

Let us first define the subspace $\Lambda_0 \subset \Lambda$ of functions with zero mean, i.e.
\begin{align} \label{def: Lambda_0}
	\Lambda_0 
	:= \left\{ \varphi_0 \in \Lambda :\ (\varphi_0, 1)_{\Gamma} = 0 \right\}.
\end{align}

We continue by constructing $\phi_\star \in \Lambda \setminus \Lambda_0$. For that, we follow \cite[Sec. 5.3]{quarteroni1999domain} and introduce $\zeta \in \Lambda$ as an interface flux with non-zero mean. For convenience, we choose $\zeta$ such that 
\begin{align} \label{eq: avg equal one}
	(\zeta, 1)_{\Gamma} = 1.
\end{align}
Any bounded $\zeta$ with this property will suffice for our purposes.
As a concrete example, we uniquely define $\zeta$ by solving a minimization problem in $H_0^1(\Gamma)$ with \eqref{eq: avg equal one} as a constraint, similar to the construction \eqref{eq: phi constraint} in Lemma~\ref{lem: inequalities}. 

Next, we test the mass conservation equation \eqref{eq: variational formulation 2nd eq} with $w = 1_{\Omega_D}$, the indicator function of $\Omega_D$. Due to the assumption $| \partial_p \Omega_D | = 0$, the divergence theorem gives us
\begin{align*}
	f_p(1_{\Omega_D})
	= -(f_D, 1)_{\Omega_D} 
	= -(\nabla \cdot (\bm{u}_D^0 + \myR_D \phi), 1)_{\Omega_D} 
	= (\phi, 1)_{\Gamma}.
\end{align*}
Using this observation, we define the function $\phi_\star \in \Lambda \setminus \Lambda_0$ such that $(\phi_\star, 1)_\Gamma = (\phi, 1)_\Gamma$ by setting
\begin{align}
	\phi_\star
	&:=  \zeta f_p(1_{\Omega_D}).
\end{align}

The next step is to pose the interface problem, similar to \eqref{eq: poincare steklov}, in this subspace: \\
Find $\phi_0 \in \Lambda_0$ such that 
\begin{align} \label{eq: poincare steklov_0}
	\langle \Sigma \phi_0, \varphi_0 \rangle &=
	\langle \chi, \varphi_0 \rangle,
	& \forall \varphi_0 &\in \Lambda_0,
\end{align}
with $\Sigma: \Lambda \to \Lambda^*$ and $\chi \in \Lambda^*$ redefined as
\begin{subequations} \label{eqs: def sigma chi 0}
\begin{align}
	\langle \Sigma \phi_0, \varphi_0 \rangle 
	&:=
	a(\bm{u}_0^0 + \myR \phi_0, \myR \varphi_0) + b(\myR \varphi_0, p_0), \\
	\langle \chi, \varphi_0 \rangle
	&:= f_u(\myR \varphi_0) - a(\bm{u}_\star^0 + \myR \varphi_\star, \myR \varphi_0) - b(\myR \varphi_0, p_\star).
\end{align}
\end{subequations}
The construction of the pairs $(\bm{u}_\star^0, p_\star)$ and $(\bm{u}_0^0, p_0)$ now require solving the Darcy subproblem with pure Neumann conditions. We emphasize that due to the nature of these problems, the pressure distributions are defined up to a constant and we therefore enforce $p_\star$ and $p_0$ to have mean zero with the use of Lagrange multipliers.

In particular, let $(\bm{u}_0^0, p_0, r_0) \in \bm{V}^0 \times W \times \mathbb{R}$ satisfy the following:
\begin{subequations} \label{eqs: aux problem u0 p0}
	\begin{align}
		a(\bm{u}_0^0, \bm{v}^0) + b(\bm{v}^0, p_0)
		&= -a(\myR \phi_0, \bm{v}^0), 
		& \forall \bm{v}^0 &\in \bm{V}^0, \\
		(r_0, w)_{\Omega_D} + b(\bm{u}_0^0, w)
		&= -b(\myR \phi_0, w),  
		& \forall w &\in W, \label{eq: conservation aux 0}\\
		(p_0, s)_{\Omega_D} &= 0,
		& \forall s &\in \mathbb{R}.
	\end{align}
\end{subequations}
Similarly, we let $(\bm{u}_\star^0, p_\star, r_\star) \in \bm{V}^0 \times W \times \mathbb{R}$ solve
\begin{subequations} \label{eqs: aux problem u* p*}
	\begin{align}
		a(\bm{u}_\star^0, \bm{v}^0) + b(\bm{v}^0, p_\star)
		&= -a(\myR \phi_\star, \bm{v}^0) + f_u(\bm{v}^0), 
		& \forall \bm{v}^0 &\in \bm{V}^0, \\
		(r_\star, w)_{\Omega_D} + b(\bm{u}_\star^0, w)
		&= - b(\myR \phi_\star, w) + f_p(w), 
		& \forall w &\in W, \label{eq: conservation aux *}\\
		(p_\star, s)_{\Omega_D} &= 0,
		& \forall s &\in \mathbb{R}.
	\end{align}
\end{subequations}
We emphasize that setting $w = 1_{\Omega_D}$ in the conservation equations \eqref{eq: conservation aux 0} and \eqref{eq: conservation aux *} yields $r_\star = r_0 = 0$. Hence, these terms have no contribution to the mass balance.

The solution to problem \eqref{eq: poincare steklov_0} allows us to construct the velocity distribution:
\begin{align}
	\bm{u} := \bm{u}_0^0 + \bm{u}_\star^0 + \myR (\phi_0 + \phi_\star). \label{eq: reconstructed flux}
\end{align}

The next step is to recover the correct pressure average in $\Omega_D$. For that, we presume that the pressure solution is given by $p = p_0 + p_\star + \bar{p}$ with $\bar{p} := c_D 1_{\Omega_D}$ for some $c_D \in \mathbb{R}$. In other words, $\bar{p}$ is zero in $\Omega_S$ and a constant $c_D$ on $\Omega_D$, to be determined next.

Using $\zeta$ from \eqref{eq: avg equal one}, we substitute this function in \eqref{eq: variational formulation 1st eq} and choose the test function $\bm{v} = \myR \zeta$:
\begin{align*}
	a(\bm{u}, \myR \zeta) + b(\myR \zeta, p_0 + p_\star + \bar{p}) = f_u(\myR \zeta).
\end{align*}
Using this relationship and the divergence theorem, we make the following two observations:
\begin{subequations} \label{eq: def c_D}
	\begin{align}
		b(\myR \zeta, \bar{p}) 
		&= 
		f_u(\myR \zeta)
		- a(\bm{u}, \myR \zeta) - b(\myR \zeta, p_0 + p_\star) 
		= \langle \chi - \Sigma \phi_0, \zeta \rangle, \\
		b(\myR \zeta, \bar{p})
		&= - (\nabla \cdot \myR \zeta, \bar{p})_{\Omega_D}
		=  c_D(\zeta, 1)_\Gamma
		= c_D.
	\end{align}
\end{subequations}
Combining these two equations yields $c_D = \langle \chi - \Sigma \phi_0, \zeta \rangle$ and we set
\begin{align}\label{eq: def p bar single}
	\bar{p} := \langle \chi - \Sigma \phi_0, \zeta \rangle 1_{\Omega_D}.
\end{align}

Finally, by setting
$p := p_0 + p_\star + \bar{p}$,
we have obtained $(\bm{u}, p) \in \bm{V} \times W$, the solution to \eqref{eq: variational formulation}. We remark that the well-posedness of \eqref{eq: poincare steklov_0} follows by the same arguments as in \Cref{thm: sigma SPD}. 

\subsubsection{Coupled Neumann Problems} 
\label{ssub:coupled_neumann_problems}

In this case, we have flux conditions prescribed on the entire boundary, i.e. $\partial \Omega = \partial_u \Omega_S \cup \partial_u \Omega_D$. 
We follow the same steps as in Section~\ref{ssub:single_neumann_problem}, and highlight the differences required to treat this case.

Let us consider a slightly more general case than \eqref{eq: bilinear forms} by including a source function $f_S$ in the Stokes subdomain. In other words, the right-hand side of the mass balance equations is given by
\begin{align}
	f_p(w) := -(f_S, w_S)_{\Omega_S} -(f_D, w_D)_{\Omega_D}.
\end{align}
By compatibility of the source function with the boundary conditions, it follows that $f_p(1) = 0$ and therefore
\begin{align}
	f_p(1_{\Omega_D})
	= (f_D, 1)_{\Omega_D}
	= -(f_S, 1)_{\Omega_S}
	= - f_p(1_{\Omega_S}).
\end{align}

Let $\Lambda_0 \subset \Lambda$ be defined as in \eqref{def: Lambda_0} and $\zeta$ as in \eqref{eq: avg equal one}. Using the same arguments as in the previous section, we define
\begin{align}
	\phi_\star
	&:=  \zeta f_p(1_{\Omega_D})
	= - \zeta f_p(1_{\Omega_S}).
\end{align}

The operators $\Sigma$ and $\chi$ are defined as in \eqref{eqs: def sigma chi 0} with the only difference being in the pairs of functions $(\bm{u}_0^0, p_0)$ and $(\bm{u}_\star^0, p_\star)$. As before, these pairs are constructing by solving the separate subproblems. Since these correspond to Neumann problems, it follows that the pressure distributions $p_0$ and $p_\star$ are defined up to a constant in each subdomain. 
We therefore enforce zero mean of these variables in each subdomain with the use of a Lagrange multiplier $s \in S$. Let $S$ be the space of piecewise constant functions given by
\begin{align}
	S := \operatorname{span}\{ 1_{\Omega_S}, 1_{\Omega_D} \}.
\end{align}

We augment problem \eqref{eqs: aux problem u0 p0} to: \\
Find $(\bm{u}_0^0, p_0, r_0) \in \bm{V}^0 \times W \times S$ such that
\begin{subequations}
	\begin{align}
		a(\bm{u}_0^0, \bm{v}^0) + b(\bm{v}^0, p_0)
		&= -a(\myR \phi_0, \bm{v}^0), 
		& \forall \bm{v}^0 &\in \bm{V}^0, \\
		(r_0, w)_{\Omega} + b(\bm{u}_0^0, w)
		&= -b(\myR \phi_0, w),  
		& \forall w &\in W,\\
		(p_0, s)_{\Omega} &= 0,
		& \forall s &\in S.
	\end{align}
\end{subequations}
Problem \eqref{eqs: aux problem u* p*} is changed analogously to produce $(\bm{u}_\star^0, p_\star, r_\star) \in \bm{V}^0 \times W \times S$. After solving the interface problem, all ingredients are available to construct the velocity $\bm{u}$ as in \eqref{eq: reconstructed flux}.

In the construction of the pressure $p$, we compute $c_D = \langle \chi - \Sigma \phi_0, \zeta \rangle$ using the same arguments as in \eqref{eq: def c_D}. Since the pressure is globally defined up to a constant, we ensure that the pressure distribution has mean zero on $\Omega$ by setting
\begin{align} \label{eq: def p bar pure}
	\bar{p} := \langle \chi - \Sigma \phi_0, \zeta \rangle \left(
	1_{\Omega_D} - \frac{| \Omega_D |}{| \Omega |}
	\right).
\end{align}
As before, we set
$p := p_0 + p_\star + \bar{p}$ and obtain the solution $(\bm{u}, p)$ of the original problem \eqref{eq: variational formulation}.

\section{Discretization} 
\label{sec:discretization}

This section presents the discretization of problem \eqref{eq: variational formulation} with the use of the Mixed Finite Element method. By introducing the interface flux as a separate variable, we derive a mortar method reminiscent of \cite{boon2018robust,nordbotten2019unified}, presented in the context of fracture flow. The focus in this section is to introduce this flux-mortar method for the coupled Stokes-Darcy problem and show its stability.

Let $\Omega_{S, h}$, $\Omega_{D, h}$, $\Gamma_h$ be shape-regular tesselations of $\Omega_S$, $\Omega_D$, and $\Gamma$, respectively. Let $\Omega_{S, h}$ and $\Omega_{D, h}$ be constructed independently and consist of simplicial or quadrangular (hexahedral in 3D) elements. Similarly, $\Gamma_h$ is a simplicial or quadrangular mesh of dimension $n - 1$, constructed according to the restrictions mentioned below. 

We impose the following three restriction on the Mixed Finite Element discretization:
\begin{enumerate}
	\item
	For the purpose of structure preservation, the finite element spaces are chosen such that
	\begin{subequations} \label{eq: inclusions}
		\begin{align}
			\bm{V}_{S, h} 
			&\subset \bm{V}_S, &
			\bm{V}_{D, h} 
			&\subset \bm{V}_D, &
			\Lambda_h 
			&\subset \Lambda, \\
			W_{S, h} 
			&\subset W_S, & 
			W_{D, h} 
			&\subset W_D.
		\end{align}
	\end{subequations}
	It is convenient, but not necessary, to define $\Gamma_h$ as the trace mesh of $\Omega_{S, h}$ and $\Lambda_h = (\bm{n} \cdot \bm{V}_{S, h})|_\Gamma$. In this case, it follows that $\Lambda_h \subset H_0^1(\Gamma) \subset H_{00}^{1/2}(\Gamma) = \Lambda$. We moreover define
	\begin{align}
		\bm{V}_{S, h}^0 
			&:= \bm{V}_{S, h} \cap \bm{V}_S^0, &
			\bm{V}_{D, h}^0 
			&:= \bm{V}_{D, h} \cap \bm{V}_D^0.
	\end{align}

	\item
	The Mixed Finite Element spaces $V_{i, h} \times W_{i, h}$ with $i \in \{S, D\}$ are chosen to form stable pairs for the Stokes and Darcy (sub)systems, respectively.
	In particular, bounded interpolation operators $\Pi_{V_i}$ exist for $i \in \{S, D\}$ such that for all $\bm{v} \in \bm{V} \cap H^\epsilon(\Omega)$ with $\epsilon > 0$:
	\begin{align} \label{eq: commutative}
		\Pi_{W_S} \nabla \cdot ((I - \Pi_{V_S}) \bm{v}_S) &= 0, & 
		\nabla \cdot (\Pi_{V_D} \bm{v}_D) &= \Pi_{W_D} \nabla \cdot \bm{v}_D.
	\end{align}
	in which $\Pi_{W_i}$ is the $L^2$-projection onto $W_{i, h}$. 
	Moreover, to ensure local mass conservation, we assume that the space of piecewise constants $(P_0)$ is contained in $W_h$.

	Examples for the Stokes subproblem include $\bm{P}_2-P_0$ in the two-dimensional case as well as the Bernardi-Raugel pair \cite{bernardi1985analysis}. 
	For the Darcy subproblem, stable choices of low-order finite elements include the Raviart-Thomas pair $RT_0-P_0$ \cite{raviart1977mixed} and the Brezzi-Douglas Marini pair $BDM_1-P_0$ \cite{brezzi1985two}. For more examples of stable Mixed Finite Element pairs, we refer the reader to \cite{boffi2013mixed}.

	\item
	For $\phi_h \in \Lambda_h$, let the discrete extension operators $\myR_{i, h}: \Lambda_h \to \bm{V}_{i, h}$ with $i \in \{S, D\}$ satisfy
	\begin{align} \label{eq: extension property h}
		(\phi_h - \bm{n} \cdot \myR_{i, h} \phi_h, \bm{n} \cdot \bm{v}_{i, h} )_{\Gamma} &= 0, & \forall \bm{v}_{i, h} &\in \bm{V}_{i, h}.
	\end{align}
	The extension operators are continuous in the sense that for $\phi_h \in \Lambda_h$, we have
	\begin{align} \label{eq: continuity R_h}
		\| \myR_{S, h} \varphi_h \|_{1, \Omega_S} &\lesssim \| \varphi_h \|_{\half, \Gamma}, &
		\| \myR_{D, h} \varphi_h \|_{0, \Omega_D} + \| \nabla \cdot \myR_D \varphi_h \|_{0, \Omega_D} &\lesssim \| \varphi_h \|_{-\half, \Gamma}.
	\end{align}
	We define $\myR_h \phi = \myR_{S, h} \oplus \myR_{D, h}$Let the mesh $\Gamma_h$ and function space $\Lambda_h$ be chosen such that the kernel of $\myR_h$ is zero:
	\begin{align}
		\myR_h \phi_h = 0 \text{ if and only if } \phi_h = 0.
	\end{align}
	We remark that this is a common restriction encountered in mortar methods (see e.g. \cite{arbogast2000mixed}) and can be satisfied by choosing $\Gamma_h$ sufficiently coarse or constructing $\Lambda_h$ using polynomials of lower order. 
\end{enumerate}

With the above restrictions in place, we define the discretizations of the combined spaces $\bm{V}$ and $W$ as
\begin{subequations}
\begin{align}
	\bm{V}_h &:= \left\{ \bm{v}_h :\ 
	\exists (\bm{v}_{S, h}^0, \varphi_h, \bm{v}_{D, h}^0) \in \bm{V}_{S, h}^0 \times \Lambda_h \times \bm{V}_{D, h}^0
	\text{ such that } \bm{v}_h|_{\Omega_i} = \bm{v}_{i, h}^0 + \myR_{i, h} \varphi_h, \text{ for } i \in \{S, D\} \right\}
	, \\
	W_h &:= W_{S, h} \times W_{D, h}.
\end{align}
\end{subequations}

As in the continuous case, the function space $\bm{V}_h$ is independent of the choice of extension operators. We remark that in the case of non-matching grids or if different polynomial orders are chosen for $\bm{V}_{S, h}$ and $\bm{V}_{D, h}$, we have $V_h \not \subset V$ due to the weaker property of $\myR_h$ in \eqref{eq: extension property h} as opposed to $\myR$ defined by \eqref{eq: extension property}. Nevertheless, a normal flux continuity is imposed in the sense that the normal trace of $\bm{v}_{S, h}$ and $\bm{v}_{D, h}$ are $L^2$ projections of a single variable $\phi_h$.

Again, the subscript $i \in \{S, D\}$ distinguishes the restrictions of $(\bm{v}, w) \in \bm{V}_h \times W_h$ to the different subdomains:
\begin{align}
	\bm{v}_{i, h} &:= \bm{v}_{i, h}^0 + \myR_{i, h} \varphi_h \in \bm{V}_{i, h}, & 
	w_{i, h} &:= w_h|_{\Omega_i} \in W_{i, h}.
\end{align}

We finish this section by formally stating the discrete problem: \\
Find the pair $(\bm{u}_h, p_h) \in \bm{V}_h \times W_h$ such that
\begin{subequations} \label{eq: variational formulation_h}
	\begin{align}
		a(\bm{u}_h, \bm{v}_h) + b(\bm{v}_h, p_h) &= f_u(\bm{v}_h),
		& \forall \bm{v}_h &\in \bm{V}_h, \\
		b(\bm{u}_h, w_h) &= f_p(w_h), 
		& \forall w_h &\in W_h.
	\end{align}
\end{subequations}
with the bilinear forms and functionals defined in \eqref{eq: bilinear forms}.

With the chosen spaces, the discretizations on $\Omega_{S, h}$ and $\Omega_{D, h}$ are stable and consistent for the Stokes and Darcy subproblems, respectively. However, in order to show stability of the method for the fully coupled problem \eqref{eq: variational formulation_h}, we briefly confirm that the relevant inequalities hold independent of the mesh parameter $h$.

\begin{lemma} \label{lem: inequalities_h}
	The following inequalities are satisfied:
	\begin{subequations}
	\begin{align}
			& \text{For } \bm{u}_h, \bm{v}_h \in \bm{V}_h: 
			&  a(\bm{u}_h, \bm{v}_h) &\lesssim \| \bm{u}_h \|_V \| \bm{v}_h \|_V. \label{ineq: a_cont_h}\\
			& \text{For } (\bm{v}_h, w_h) \in \bm{V}_h \times W_h: 
			&  b(\bm{v}_h, w_h) &\lesssim \| \bm{v}_h \|_V \| w_h \|_W. \label{ineq: b_cont_h}\\
			& \text{For } \bm{v}_h \in \bm{V}_h \text{ with } b(\bm{v}_h, w_h) = 0 \ \forall w_h \in W_h: 
			&  a(\bm{v}_h, \bm{v}_h) &\gtrsim \| \bm{v}_h \|_V^2. \label{ineq: a_coercive_h}\\
			& \text{For } w_h \in W_h, \ \exists \bm{v}_h \in \bm{V}_h \text{ with } \bm{v}_h \ne 0, \text{ such that}: 
			&  b(\bm{v}_h, w_h) &\gtrsim \| \bm{v}_h \|_V \| w_h \|_W.\label{ineq: b_infsup_h}
		\end{align}
	\end{subequations}
	\end{lemma}
	\begin{proof}
		Inequalities \eqref{ineq: a_cont_h} and \eqref{ineq: b_cont_h} follow using the same arguments as \eqref{ineqs: continuity} in \Cref{lem: inequalities}. Continuing with \eqref{ineq: a_coercive_h}, we note from \eqref{eq: commutative} that $b(\bm{v}_h, w_h) = 0$ implies
		\begin{align*}
			0 = \Pi_{W_D} \nabla \cdot \bm{v}_{D, h} = \nabla \cdot (\Pi_{V_D} \bm{v}_{D, h}) = \nabla \cdot \bm{v}_{D, h}.
		\end{align*}
		Hence, the same derivation as \eqref{eq: proof 3.5c} is followed to give us \eqref{ineq: a_coercive_h}. 

		For the final inequality, we consider $w_h \in W_h$ given and follow the strategy of \Cref{lem: inequalities}. First, we set up a minimization problem in $\Lambda_h$ analogous to \eqref{eq: phi constraint} to obtain a bounded $\phi_h \in \Lambda_h$ such that
		\begin{subequations}
		\begin{align}
			\| \phi_h \|_{1, \Gamma} &\lesssim \| K w_{D, h} \|_{0, \Omega_D}, \\
			(\nabla \cdot \myR_{D, h} \phi_h, 1)_{\Omega_D} &=
			(- \bm{n} \cdot \myR_{D, h} \phi_h, 1)_{\Gamma} = 
			(- \phi_h, 1)_{\Gamma} = 
			- (K w_{D, h}, 1)_{\Omega_D}. \label{eq: compatibility of phi_h}
		\end{align}
		\end{subequations}

		Next, we note that $w_h \in W_h \subset W$. In turn, we use the auxiliary problems \eqref{eqs: aux prob p_S} and \eqref{eqs: aux prob p_D} to construct $\bm{v}_S^0 \in \bm{V}_S^0$ and $\bm{v}_D^0 \in \bm{V}_D^0$ such that
		\begin{subequations}
		\begin{align}
			- \nabla \cdot \bm{v}_S^0 &= \mu^{-1} w_{S, h} + \nabla \cdot \myR_{S, h} \phi_h, \\
			- \nabla \cdot \bm{v}_D^0 &= K w_{D, h} + \nabla \cdot \myR_{D, h} \phi_h, \\
			\| \bm{v}_S^0 \|_{1, \Omega_S} &\lesssim 
			\| \mu^{-1} w_{S, h} \|_{0, \Omega_S}  + \| \nabla \cdot \myR_{S, h} \phi_h \|_{0, \Omega_S}, \\
			\| \bm{v}_D^0 \|_{1, \Omega_D} &\lesssim 
			\| K w_{D, h} \|_{0, \Omega_D}  + \| \nabla \cdot \myR_{D, h} \phi_h \|_{0, \Omega_D}.
		\end{align}
		\end{subequations}

		We then employ the interpolation operators from \eqref{eq: commutative} to create $\bm{v}_{S, h}^0 = \Pi_{V_S} \bm{v}_S^0$ and $\bm{v}_{D, h}^0 = \Pi_{V_D}\bm{v}_D^0$. Using the commutative properties, we obtain
		\begin{align*}
			b(\bm{v}_h, w_h) &
			= - \sum_{i \in \{S, D\}} (\nabla \cdot (\Pi_{V_i} \bm{v}_i^0 + \myR_{i, h} \phi_h), w_{i, h})_{\Omega_i}
			= (\mu^{-1} w_{S, h}, w_{S, h})_{\Omega_S} + (K w_{D, h}, w_{D, h})_{\Omega_D}
			= \| w_h \|_W^2.
		\end{align*}
		Moreover, by the boundedness of these interpolation operators, we have
		\begin{align*}
			\| \bm{v}_h \|_V 
			\le \| \bm{v}_h^0 \|_V + \| \myR_h \phi_h \|_V
			\lesssim \| \bm{v}^0 \|_V + \| \phi_h \|_\Lambda
			\lesssim \| w_h \|_W,
		\end{align*}
		proving the final inequality \eqref{ineq: b_infsup_h}.
	\end{proof}

	\begin{theorem}
		If the three conditions presented at the beginning of this section are satisfied, then the discretization method is stable, i.e. a unique solution $(\bm{u}_h, p_h) \in \bm{V}_h \times W_h$ exists for \eqref{eq: variational formulation_h} satisfying
	\begin{align}
		\| \bm{u}_h \|_V + \| p_h \|_W 
		\lesssim 
		\| \mu^{-\half} \bm{f}_S \|_{-1, \Omega_S}
		+ \| K^{-\half} f_D \|_{0, \Omega_D}
		+ \| K^{\half} g_p \|_{\half, \partial_p \Omega_D}.
	\end{align}
	\end{theorem}
	\begin{proof}
		This result follows from \Cref{lem: inequalities_h}, the continuity of the right-hand side from \Cref{thm: well-posedness}, and saddle point theory.
	\end{proof}

\section{Iterative Solution Method} 
\label{sec:iterative_solvers}

With well-posedness of the discrete system shown in the previous section, we continue by constructing an efficient solution method to solve the coupled system in an iterative manner. The scheme is introduced according to three steps. We first present the discrete Steklov-Poincar\'e system that we aim to solve using a Krylov subspace method. Second, a parameter-robust preconditioner is introduced for the reduced system. The third step combines these two ideas to form an iterative method that respects mass conservation at each iteration.

\subsection{Discrete Steklov-Poincar\'e System} 
\label{sub:discrete_poincar}

Similar to the continuous case in Section~\ref{sec:the_steklov_poincare_system}, we reduce the problem to the interface flux variables $\phi_h \in \Lambda_h$. The reduced system is a direct translation of \eqref{eq: poincare steklov} to the discrete setting: \\
Find $\phi_h \in \Lambda_h$ such that
\begin{align} \label{eq: poincare steklov_h}
	\lang \Sigma_h \phi_h, \varphi_h \rang &= \lang \chi_h, \varphi_h \rang, &
	\forall \varphi_h &\in \Lambda_h.
\end{align}

To ease the implementation of the scheme, we focus particularly on the structure of the operator $\Sigma_h$. For that, we note that the space $\bm{V}_h^0$ can be decomposed orthogonally into $\bm{V}_{S, h}^0 \times \bm{V}_{D, h}^0$ and a similar decomposition holds for $W_h$. The aim is to propose a solution method which exploits this property. For brevity, the subscript $h$ is omitted on all variables and operators, keeping in mind that the remainder of this section concerns the discretized setting.

Let us rewrite the bilinear forms $a$ and $b$ in terms of duality pairings, thereby revealing the matrix structure of the problem. For that, we group the terms according to the subdomains and let the operators $A_i: \bm{V}_{i, h} \to \bm{V}_{i, h}^*$ and $B_i:\bm{V}_{i, h} \to W_{i, h}^*$ be defined for $i \in \{S, D\}$ such that
\begin{align*}
	\lang A_S \bm{u}_S, \bm{v}_S \rang  
	&= (\mu \varepsilon(\bm{u}_S), \varepsilon(\bm{v}_S))_{\Omega_S} 
		+ (\beta \bm{\tau} \cdot \bm{u}_S, \bm{\tau} \cdot \bm{v}_S )_{\Gamma}, \\
	\lang A_D \bm{u}_D, \bm{v}_D \rang  
	&= (K^{-1} \bm{u}_D, \bm{v}_D)_{\Omega_D}, \\
	\lang B_S \bm{u}_S, w_S \rang 
	&= -(\nabla \cdot \bm{u}_S, w_S)_{\Omega_S}, \\
	\lang B_D \bm{u}_D, w_D \rang 
	&= -(\nabla \cdot \bm{u}_D, w_D)_{\Omega_D}.
\end{align*}

Let $A_{i, 0}$ and $B_{i, 0}$ be the respective restrictions of the above to the subspace $\bm{V}_{i, h}^0$.
With these operators and the decomposition $\bm{u}_i = \bm{u}_i + \myR_i \phi$ for the trial and test functions, problem \eqref{eq: variational formulation} attains the following matrix form:
\begin{align} \label{eq: matrix form}
	\begin{bmatrix}
		A_{S, 0}       & B_{S, 0}^T     &                &                & A_S \myR_S \\[5pt]
		B_{S, 0}       &                &                &                & B_S \myR_S \\[5pt]
		               &                & A_{D, 0}       & B_{D, 0}^T     & A_D \myR_D \\[5pt]
		               &                & B_{D, 0}       &                & B_D \myR_D \\[5pt]
		(A_S \myR_S)^T & (B_S \myR_S)^T & (A_D \myR_D)^T & (B_D \myR_D)^T & \sum_i \myR_i^T A_i \myR_i
	\end{bmatrix}
	\begin{bmatrix}
		\bm{u}_S^0 \\[5pt]
		p_S \\[5pt]
		\bm{u}_D^0 \\[5pt]
		p_D \\[5pt]
		\phi
	\end{bmatrix}
	&= 
	\begin{bmatrix}
		f_{S, u} \\[5pt]
		f_{S, p} \\[5pt]
		f_{D, u} \\[5pt]
		f_{D, p} \\[5pt]
		f_{\phi} 
	\end{bmatrix}.
\end{align}
In practice, the discrete extension operators are chosen to only have support in the elements adjacent to $\Gamma$, leading to a favorable sparsity pattern.
We moreover note that the final row corresponds to test functions $\varphi \in \Lambda_h$.
The right-hand side of \eqref{eq: matrix form} is defined such that
\begin{align}
	\lang f_{S, u}, \bm{v}_S^0 \rang 
	+ \lang f_{S, p}, w_S \rang 
	+ \lang f_{D, u}, \bm{v}_D^0 \rang
	+ \lang f_{D, p}, w_D \rang
	+ \lang f_{\phi}, \varphi \rang
	&=
	f_u(\bm{v}) + f_p(w), &
	\forall (\bm{v}, w) &\in \bm{V}_h \times W_h.
\end{align}

The discrete Steklov-Poincar\'e system is obtained by taking a Schur-complement of this system. In particular, we obtain $\Sigma_h$ and the right-hand side $\chi_h$ as
\begin{subequations}
\begin{align} 
	\label{eq: def sigma_h}
	\Sigma_h &:= \sum_{i \in \{S, D\} } 
	\myR_i^T A_i \myR_i
	- \myR_i^T [A_i \ B_i^T] 
	\begin{bmatrix}
	A_{i, 0}   & B_{i, 0}^T \\[5pt]
	B_{i, 0}   & \end{bmatrix}^{-1}
	\begin{bmatrix}
		A_i \\[5pt]
		B_i
	\end{bmatrix}
	\myR_i, \\
	\chi_h &:= f_{\phi}
	- \sum_{i \in \{S, D\} } \myR_i^T [A_i \ B_i^T] 
	\begin{bmatrix}
	A_{i,0}   & B_{i, 0}^T \\[5pt]
	B_{i, 0}  & \end{bmatrix}^{-1}
	\begin{bmatrix}
		f_{u, i} \\[5pt]
		f_{p, i}
	\end{bmatrix}.
	\label{eq: def chi_h}
\end{align}
\end{subequations}

We employ a Krylov subspace method to solve \eqref{eq: poincare steklov_h} iteratively, thereby avoiding the computationally costly assembly of $\Sigma_h$. In order to obtain a parameter-robust iterative method, the next step is to introduce an appropriate preconditioner, as presented in the next section. 

\subsection{Parameter-Robust Preconditioning}
\label{sub:parameter_robust_preconditioning}

In this section, we construct the preconditioner such that the resulting iterative method is robust with respect to the material parameters ($K$ and $\mu$) and the mesh size ($h$). For that, we use the parameter-dependent norm $\| \cdot \|_\Lambda$ from \eqref{eq: norm phi} and follow the framework presented by \cite{mardal2011preconditioning} to form a norm-equivalent preconditioner. In particular, we use the following result from that work:
\begin{lemma}
	Given a bounded, symmetric, positive-definite operator $\Sigma: \Lambda \to \Lambda^*$ 
	and a symmetric positive definite operator $\mathcal{P}: \Lambda^* \to \Lambda$. If the induced norm $\| \phi \|_{\mathcal{P}^{-1}}^2 := \lang \mathcal{P}^{-1} \phi, \phi \rang$ satisfies
	\begin{align} \label{eq: norm equivalence precond}
		\| \phi \|_{\Lambda}^2 \lesssim
		\| \phi \|_{\mathcal{P}^{-1}}^2 \lesssim
		\| \phi \|_{\Lambda}^2,
	\end{align}
	then $\mathcal{P}$ is a robust preconditioner in the sense that the condition number of $\mathcal{P} \Sigma$ is bounded independent of the material and mesh parameters.
\end{lemma}

Note that symmetry of $\Sigma_h$ is apparent from \eqref{eq: def sigma_h}. Positive definiteness follows using the same arguments as in \Cref{thm: sigma SPD}. The next step is therefore to create an operator $\mathcal{P}^{-1}$ that generates a norm which is equivalent to $\| \cdot \|_\Lambda$ on $\Lambda_h$. Recall from \eqref{eq: norm phi} that $\| \cdot \|_\Lambda$ is composed of fractional Sobolev norms. The key idea is to introduce a matrix $\mathsf{H}(s)$ that induces a norm which is equivalent to $H^s(\Gamma)$ for $s = \pm \half$. We apply the strategy explained in \cite{kuchta2016preconditioners} to achieve this, of which a short description follows.

For given basis $\{\phi_i\}_{i = 1}^{n_\Lambda} \in \Lambda_h$ with $n_\Lambda$ the dimension of $\Lambda_h$, let the mass matrix $\mathsf{M}$ and stiffness matrix $\mathsf{A}$ be defined such that
\begin{align}
	\mathsf{M}_{ij} &:= ( \phi_j, \phi_i )_{\Gamma}, & 
	\mathsf{A}_{ij} &:= ( \nabla_\Gamma \phi_j, \nabla_\Gamma \phi_i )_{\Gamma}.
\end{align}

Then, a complete set of eigenvectors $\mathsf{v}_i \in \mathbb{R}^{n_\Lambda}$ and eigenvalues $\lambda_i \in \mathbb{R}$ exist solving the generalized eigenvalue problem 
\begin{align} \label{eq: generalized eigenvalue problem}
	\mathsf{A} \mathsf{v}_i = \lambda_i \mathsf{M} \mathsf{v}_i.
\end{align}
The eigenvectors satisfy $\mathsf{v}_i^\mathsf{T} \mathsf{M} \mathsf{v}_j = \delta_{ij}$ with $\delta_{ij}$ the Kronecker delta function. Let the diagonal matrix $\mathsf{\Lambda} := \operatorname{diag}([\lambda_i]_{i = 1}^{n_\Lambda})$ and let $\mathsf{V}$ be the matrix with $\mathsf{v}_i$ as its columns. The following eigendecomposition then holds:
\begin{align}
	\mathsf{A} = \mathsf{(MV) \Lambda (MV)^T}.
\end{align}

Using the matrices $\mathsf{M}$, $\mathsf{V}$, and $\mathsf{\Lambda}$, we define the operator $\mathsf{H}: \mathbb{R} \to \mathbb{R}^{n_\Lambda \times n_\Lambda}$ as
\begin{align}
	\mathsf{H}(s) = \mathsf{(MV) \Lambda}^s \mathsf{(MV)^T}.
\end{align}
An advantage of this construction is that its inverse can be directly computed as $\mathsf{H}(s)^{-1} = \mathsf{V \Lambda}^{-s} \mathsf{V^T}$ due to $\mathsf{V^TMV = I}$. Next, we emphasize that $\mathsf{H}(0) = \mathsf{M}$ and $\mathsf{H}(1) = \mathsf{A}$, i.e. the discrete $L^2(\Gamma)$ and $H^1(\Gamma)$ norms on $\Lambda_h$ are generated for $s = 0$ and $s = 1$, respectively. As a generalization, the norm induced by the matrix $\mathsf{H}(s)$ is equivalent to the $H^s(\Gamma)$ norm on the discrete space $\Lambda_h$ \cite{kuchta2016preconditioners}. In other words, 
\begin{align}
	\| \phi \|_{s, \Gamma}^2
	\lesssim (\pi \phi)^T \mathsf{H}(s) (\pi \phi)
	\lesssim \| \phi \|_{s, \Gamma}^2,
\end{align}
in which $\pi$ is the representation operator in the basis $\{\phi_i\}_{i = 1}^{n_\Lambda}$.

Next, we use these tools to define our preconditioner following the strategy of \cite{mardal2011preconditioning}. The operator $\mathsf{P}^{-1}: \mathbb{R}^{n_{\Lambda}} \to \mathbb{R}^{n_{\Lambda}}$ is defined according to the norm $\| \cdot \|_{\Lambda}$ from \eqref{eq: norm phi}:
\begin{align}
	\mathsf{P}^{-1} := \mu \mathsf{H} \left(\tfrac{1}{2}\right) + K^{-1} \mathsf{H} \left(-\tfrac{1}{2}\right).
\end{align}

Defining $\mathcal{P}^{-1} := \pi^T \mathsf{P}^{-1} \pi$, we obtain the equivalence relation \eqref{eq: norm equivalence precond} by construction. In turn, the inverse operator $\mathcal{P}$ is an optimal preconditioner for the system $\eqref{eq: poincare steklov_h}$. The matrix $\mathsf{P}$ is explicitly computed using the properties of $\mathsf{V}$ and $\mathsf{M}$:
\begin{align} \label{eq: preconditioner}
	\mathsf{P} =  
	\left(
	\mu \mathsf{H} \left(\tfrac{1}{2}\right) + K^{-1} \mathsf{H} \left(-\tfrac{1}{2}\right)
	\right)^{-1}
	=
	\mathsf{V} \left(
	\mu \mathsf{\Lambda}^{\half} + K^{-1} \mathsf{\Lambda}^{-\half}
	\right)^{-1}
	\mathsf{V^T}.
\end{align}

\subsection{An Iterative Method Respecting Mass Conservation} 
\label{sub:a_conservative_method}

The Steklov-Poincar\'e system from \cref{sub:discrete_poincar} and the preconditioner from \cref{sub:parameter_robust_preconditioning} form the main ingredients of the iterative scheme proposed next. As mentioned before, we aim to use a Krylov subspace method on the reduced system \eqref{eq: poincare steklov_h}. We turn to the Generalized Minimal Residual (GMRes) method \cite{saad1986gmres} and propose the scheme we refer to as \Cref{alg: GMRes}, described below.

\begin{algorithm}[ht]
	\caption{}
	\label{alg: GMRes}
	\begin{enumerate}
		\item Set the tolerance $\epsilon > 0$, choose an initial guess $\phi_h^0 \in \Lambda_h$, and construct the right-hand side $\chi_h$ from \eqref{eq: def chi_h} and $\mathsf{P}$ from \eqref{eq: preconditioner}.
		\item \label{step: SD solves}
		Using $\mathsf{P}$ as a preconditioner, apply GMRes to the discrete Steklov-Poincar\'e system \eqref{eq: poincare steklov_h} until the relative, preconditioned residual is smaller than $\epsilon$. This involves solving a Stokes and a Darcy subproblem in \eqref{eq: def sigma_h} at each iteration.
		\item \label{step: reconstruction}
		Construct $(\bm{u}_{S, h}, p_{S, h}) \in \bm{V}_{S, h} \times W_{S, h}$ and $(\bm{u}_{D, h}, p_{D, h}) \in \bm{V}_{D, h} \times W_{D, h}$ by solving the independent Stokes and Darcy subproblems with $\phi_h$ as the normal flux on $\Gamma$.
		\item In the case of Neumann problems, reconstruct the mean of the pressure in $\Omega_D$ using \eqref{eq: def p bar single} or \eqref{eq: def p bar pure}.
	\end{enumerate}
\end{algorithm}

We make three observations concerning this algorithm. Most importantly, the solution $(\bm{u}_h, p_h)$ produced by \Cref{alg: GMRes} conserves mass locally, independent of the number of GMRes iterations. In particular, 
the definition of the space $\bm{V}_h$ ensures that no mass is lost across the interface. Moreover, with the flux $\phi_h$ given on the interface, the reconstruction in step \eqref{step: reconstruction} ensures that mass is conserved in each subdomain.

Secondly, we emphasize that the solves for the Stokes and Darcy subproblems in step \ref{step: SD solves} can be performed in parallel by optimized solvers. Moreover, the preconditioner is local to the interface and is agnostic to the choice of extension operators and chosen discretization methods. In turn, this scheme is applicable to a wide range of well-established ``legacy'' codes tailored to solving Stokes and Darcy flow problems. 

Third, we have made the implicit assumption that obtaining $\mathsf{P}$ by solving \eqref{eq: generalized eigenvalue problem} is computationally feasible. This is typically the case if the dimension of the interface space $\Lambda_h$ is sufficiently small. The generalized eigenvalue problem is solved once in an a priori, or ``off-line'', stage and the assembled matrix $\mathsf{P}$ is then applied in each iteration of the GMRes method.

\section{Numerical Results} 
\label{sec:numerical_results}

In this section, we present numerical experiments that verify the theoretical results presented above. By setting up artificial coupled Stokes-Darcy problems in two dimension, we investigate the dependency of \Cref{alg: GMRes} on physical and discretization parameters in Section~\ref{sub:parameter_robustness}. Afterward, Section~\ref{sub:comparison_to_NN_method} presents a comparison of the proposed scheme to a Neumann-Neumann method.

\subsection{Parameter Robustness}
\label{sub:parameter_robustness}

Let the subdomains be given by $\Omega_S := (0, 1) \times (0, 1)$, $\Omega_D := (0, 1) \times (-1, 0)$, and $\Gamma := [0, 1] \times \{ 0 \}$. Two test cases are considered, defined by different boundary conditions. The first concerns the setting in which both the Stokes and Darcy subproblems are well-posed. On the other hand, test case 2 illustrates the scenario in which a pure Neumann problem is imposed on the porous medium. 

For test case 1, let $\partial_\sigma \Omega_S$ be the top boundary ($x_2 = 1$) and $\partial_u \Omega_D$ the bottom boundary ($x_2 = -1$). The remaining portions of the boundary $\partial \Omega$ form $\partial_u \Omega_S$ and $\partial_p \Omega_D$. On $\partial_u \Omega_S$, zero velocity is imposed as described by \eqref{eq: BC essential}. The pressure data is set to $g_p(x_1, x_2) := x_2$ on $\partial_p \Omega_D$ to stimulate a flow field that infiltrates the porous medium.

Test case 2 simulates parallel flow over a porous medium. We impose no-flux conditions on $\partial \Omega_D \setminus \Gamma$, thereby ensuring that all mass transfer to and from the porous medium occurs at the interface $\Gamma$. The flow is stimulated by prescribing the velocity at the left and right boundaries of $\Omega_S$ using the parabolic profile $\bm{u}_S(x_1, x_2) := [0, x_2 (2 - x_2)]$. As in test case 1, the top boundary represents $\partial_\sigma \Omega_S$, at which a zero stress is prescribed.

Both test cases consider zero body force and mass source, i.e. $\bm{f}_S := 0$ and $f_D := 0$. Moreover, we set the parameter $\alpha$ in the Beavers-Joseph-Saffman condition to zero for simplicity. 

The meshes $\Omega_{S, h}$ and $\Omega_{D, h}$ are chosen to be matching at $\Gamma$ and we set $\Gamma_h$ as the coinciding trace mesh. Following \Cref{sec:discretization}, we choose the Mixed Finite Element method in both subdomains, implemented with the use of FEniCS \cite{logg2012automated}. The spaces are given by a vector of quadratic Lagrange elements $(\bm{P}_2)$ for the Stokes velocity $\bm{V}_{S, h}$ and lowest order Raviart-Thomas elements $(RT_0)$ for the Darcy velocity $\bm{V}_{D, h}$. The pressure space $W_h$ is given by piecewise constants $(P_0)$. The interface space $\Lambda_h$ is chosen to be the normal trace of $\bm{V}_{S, h}$ on $\Gamma_h$, and therefore consists of quadratic Lagrange elements.

For the sake of efficiency, the matrices $\mathsf{V}$ and $\mathsf{\Lambda}$ used in the preconditioner $\mathsf{P}$ are computed a priori. Moreover, we pre-compute the $\mathsf{LU}$-decompositions of the Darcy and Stokes subproblems. These decompositions serve as surrogates for optimized ``legacy'' codes. The iterative solver is terminated when a relative residual of $\epsilon = 10^{-6}$ is reached, i.e. when the ratio of Euclidean norms of the preconditioned residual and the preconditioned right-hand side is smaller than $\epsilon$. 

We first consider the dependency of the iterative solver on the mesh size. We set unit viscosity and permeability and start with a coarse mesh with $h = 1/8$. The mesh is refined four times and at each refinement, the number of iterations in \Cref{alg: GMRes} is reported. The results for both test cases are shown in \Cref{tab: Robustness}. 

The results show that the number of iterations is robust with respect to the mesh size. Moreover, the second and third columns indicate the reduction from a fully coupled problem of size $n_{total}$ to a significantly smaller interface problem of size $n_\Lambda$. As shown in \Cref{sec:the_steklov_poincare_system}, this interface problem is symmetric and positive definite.

\begin{table}[ht]
	\caption{The number of iterations necessary to reach the given tolerance with respect to the mesh size. The material parameters are given by $\kappa = \mu = 1$. }
	\label{tab: Robustness}
	\centering
	\begin{tabular}{|r|r|r|c|c|}
		\hline

		\hline
			$1 / h$ &
			$n_{total}$ &
			$n_\Lambda$ &
			Case 1 &
			Case 2 \\
		\hline
			  8 &   1,042 &  15 & 8 & 6 \\
			 16 &   4,002 &  31 & 9 & 8 \\
			 32 &  15,682 &  63 & 8 & 9 \\
			 64 &  62,082 & 127 & 8 & 9 \\
			128 & 247,042 & 255 & 8 & 8 \\
		\hline

		\hline
	\end{tabular}
\end{table}

We investigate the robustness of \Cref{alg: GMRes} with respect to physical parameters by varying both the (scalar) permeability $\kappa$ and the viscosity $\mu$ over a range of eight orders of magnitude. The number of iterations is reported in Table~\ref{tab: Robustness parameters}. 

It is apparent that the scheme is robust with respect to both parameters, reaching the desired tolerance within a maximum of eleven iterations for the two test cases. Minor deviations in the iteration numbers can be observed for low permeabilities. The scheme may require an extra iteration in that case due to the higher sensitivity of the Darcy subproblem on flux boundary data. 

\begin{table}[ht]
	\caption{Number of iterations necessary to reach the tolerance with respect to the physical parameters. For both cases, the mesh size is set to $h = 1/64$.}
	\label{tab: Robustness parameters}
	\centering
	\begin{tabular}{|r|r|rrrrr|}
		\hline

		\hline
		\multicolumn{2}{|c|}{\multirow{2}{*}{Case 1}}
		& 
		\multicolumn{5}{c|}{log$_{10}(\mu)$} \\
		\cline{3-7}
		\multicolumn{2}{|c|}{} 
		& $-4$ & $-2$ & $\phantom{-}0$ & $\phantom{-}2$ & $\phantom{-}4$ \\
		\hline
			\multirow{5}{*}{log$_{10}(\kappa)$}
			&  4 & 8 & 8 & 8 & 8 & 8 \\
			&  2 & 8 & 8 & 8 & 8 & 8 \\
			&  0 & 8 & 8 & 8 & 8 & 8 \\
			& $-2$ & 7 & 7 & 7 & 7 & 7 \\
			& $-4$ & 7 & 7 & 7 & 7 & 7 \\
		\hline

		\hline
	\end{tabular}
	\hspace{50 pt}
	\begin{tabular}{|r|r|rrrrr|}
		\hline

		\hline
		\multicolumn{2}{|c|}{\multirow{2}{*}{Case 2}}
		& 
		\multicolumn{5}{c|}{log$_{10}(\mu)$} \\
		\cline{3-7}
		\multicolumn{2}{|c|}{} 
		& $-4$ & $-2$ & $\phantom{-}0$ & $\phantom{-}2$ & $\phantom{-}4$ \\
		\hline
			\multirow{5}{*}{log$_{10}(\kappa)$}
			&  4 &  9 &  9 &  9 &  9 &  9 \\
			&  2 &  9 &  9 &  9 &  9 &  9 \\
			&  0 &  9 &  9 &  9 &  9 &  9 \\
			& $-2$ &  9 &  9 &  9 &  9 &  9 \\
			& $-4$ & 11 & 11 & 11 & 11 & 10 \\
		\hline
 
		\hline
	\end{tabular}
\end{table}

\subsection{Comparison to a Neumann-Neumann Method} 
\label{sub:comparison_to_NN_method}

In order to compare the performance of Algorithm~\ref{alg: GMRes} to more conventional domain decomposition methods, we consider the closely-related Neumann-Neumann method. This method, as remarked in \cite[Remark 3.1]{discacciati2007robin}, solves the Steklov-Poincar\'e sytem \eqref{eq: poincare steklov_h} in the following, iterative manner. Given the current residual, we solve the Stokes and Darcy subproblems by interpreting this residual as a normal stress, respectively pressure, boundary condition. The computed fluxes normal to $\Gamma$ then update $\phi$ through the following operator:
\begin{align}
	\mathcal{P}_{NN} := \Sigma_{S, h}^{-1} + \Sigma_{D, h}^{-1},
\end{align}
with $\Sigma_{S, h} + \Sigma_{D, h} = \Sigma_h$ the decomposition in \eqref{eq: def sigma_h}. 

Noting that $\mathcal{P}_{NN}$ is an approximation of $\Sigma_h^{-1}$, we define the Neumann-Neumann method by replacing the preconditioner $\mathsf{P}$ by $\mathcal{P}_{NN}$ in Algorithm~\ref{alg: GMRes}. Moreover, we choose the same Krylov subspace method (GMRes) and stopping criterion, in order to make the comparison as fair as possible. 

We consider the numerical experiment from \cite{discacciati2005iterative,discacciati2007robin} posed on $\Omega := (0, 1) \times (0, 2)$ with $\Omega_S := (0, 1) \times (1, 2)$, $\Omega_D := (0, 1) \times (0, 1)$, and $\Gamma = (0, 1) \times \{ 1 \}$.
The solution is given by $\bm{u}_S := \left[(x_2 - 1)^2, \  x_1(x_1 - 1)\right]$, $p_S = \mu (x_1 + x_2 - 1) + (3K)^{-1}$, and $p_D = \left(x_1(1 - x_1)(x_2 - 1) + \frac13 x_2^3 - x_2^2 + x_2 \right) K^{-1} + \mu x_1$. The boundary conditions are chosen to comply with this solution and we impose the pressure on $\partial \Omega_D \setminus \Gamma$, the normal stress on the top boundary, and the velocity on the remaining boundaries of $\Omega_S$. Finally, the Beavers-Joseph-Saffman condition is replaced by a no-slip condition for the tangential Stokes velocity on $\Gamma$.

Using the same Mixed Finite Element discretization as in the previous section, we vary the material and discretization parameters and report the number of iterations necessary to reach a relative residual of $\epsilon = 10^{-6}$ to the preconditioned problem. The results are presented in Table~\ref{tab: comparison with NN}.

\begin{table}[ht]
	
	\caption{Iteration counts of the proposed scheme compared to a Neumann-Neumann method. The initial mesh has $h_0 = 1/7$ and each refinement is such that $h_{i + 1} = h_i/2$.}
	\label{tab: comparison with NN}
	\centering
	\begin{tabular}{|c|c|rrrrr|rrrrr|}
		\hline

		\hline
		\multirow{2}{*}{log$_{10}(\mu)$} & 
		\multirow{2}{*}{log$_{10}(K)$} & 
		\multicolumn{5}{c|}{Neumann-Neumann} &
		\multicolumn{5}{c|}{Algorithm~\ref{alg: GMRes}} \\
		\cline{3-12}
		& &
		$h_0$ & $h_1$ & $h_2$ & $h_3$ & $h_4$ &
		$h_0$ & $h_1$ & $h_2$ & $h_3$ & $h_4$ \\
		\hline			
		\multirow{3}{*}{$\phantom{-}0$}
		& $\phantom{-}0$ & 3 & 3 & 3 & 3 & 3 & 8 & 8 & 8 & 8 & 8 \\
		& $-1$ & 5 & 5 & 5 & 5 & 5 & 8 & 8 & 8 & 8 & 8 \\
		& $-2$ & 7 & 7 & 8 & 8 & 8 & 7 & 7 & 7 & 8 & 8 \\
		\hline 
		\multirow{3}{*}{$-1$}
		& $\phantom{-}0$ & 5 & 5 & 5 & 5 & 5 & 8 & 8 & 8 & 8 & 8 \\
		& $-1$ & 7 & 7 & 8 & 8 & 8 & 7 & 7 & 7 & 8 & 8 \\
		& $-2$ & 8 & 9 & 12 & 16 & 13 & 10 & 9 & 8 & 7 & 7 \\
		\hline 
		\multirow{3}{*}{$-2$}
		& $\phantom{-}0$ & 7 & 7 & 8 & 8 & 8 & 7 & 7 & 7 & 8 & 8 \\
		& $-1$ & 8 & 9 & 12 & 16 & 13 & 10 & 9 & 8 & 7 & 7 \\
		& $-2$ & 14 & 9 & 18 & 25 & 29 & 13 & 14 & 12 & 8 & 7 \\
		\hline

		\hline
	\end{tabular}

\end{table}

We observe that the two methods behave opposite as the mesh size decreases. Whereas the Neumann-Neumann method requires more iterations for finer grids, the performance of our proposed scheme appears to improve, requiring only eight iterations in all cases on the finest levels.

In general, Algorithm~\ref{alg: GMRes} outperforms the Neumann-Neumann method in terms of robustness, with the only deviation occurring in the case of a low permeability and a coarse grid. The Neumann-Neumann method appears more sensitive to both material and discretization parameters and only converges faster for material parameters close to unity.

In terms of computational cost, we emphasize that our proposed scheme requires an off-line computation to construct the preconditioner and contains a solve for the Stokes and Darcy subproblems at each iteration. On the other hand, the Neumann-Neumann method requires an additional solve for the subproblems in the preconditioner $\mathcal{P}_{NN}$. These additional solves will likely become prohibitively expensive for finer grids, since each solve is more costly and more iterations become necessary. Thus, if the preconditioner $\mathsf{P}$ can be formed efficiently, then Algorithm~\ref{alg: GMRes} forms an attractive alternative for such problems.

Although these results do not allow for a thorough, quantitative comparison with the Robin-Robin methods presented in \cite{discacciati2007robin}, we do make an important, qualitative observation. In particular, our proposed method does not require setting any acceleration parameters and its performance is a direct consequence of the constructed preconditioner. This is advantageous because finding the optimal values for such parameters can be a non-trivial task.

\section{Conclusion}
\label{sec:conclusions}

In this work, we proposed an iterative method for solving coupled Stokes-Darcy problems that retains local mass conservation at each iteration. By introducing the normal flux at the interface, the original problem is reduced to a smaller problem concerning only this variable. Through a priori analysis with the use of weighted norms, a preconditioner is formed to ensure that the scheme is robust with respect to physical and discretization parameters.

Future research will focus on four main ideas. First, we are interesting in investigating the application of this scheme on different discretization methods, including the pairing of a MPFA finite volume method with the MAC-scheme as in \cite{schneider2020coupling}. 

Secondly, we note that the use of non-matching grids forms another natural extension. In that case, we aim to investigate how such a mismatch affects the discretization error. However, such analysis heavily depends on the chosen discretization method and is therefore reserved as a topic for future investigation.

Third, the generalization of these ideas to non-linear problems forms another area of our interest. By considering the Navier-Stokes equations in the free-flow subdomain, for example, the reduction to an interface problem will inherit the non-linearity. An iterative method that solves this reduced problem may benefit from a similarly constructed preconditioner. 

Finally, as remarked in Section~\ref{sub:a_conservative_method}, we are under the assumption that the generalized eigenvalue problem \eqref{eq: generalized eigenvalue problem} can be solved efficiently. However, if the assembly of $\mathsf{P}$ is too costly, then more efficient, spectrally equivalent preconditioners are required. A promising example may be to employ the recent work on multi-grid preconditioners for fractional diffusion problems \cite{baerland2019multigrid}.

To conclude, we have presented this iterative method in a basic setting so that it may form a foundation for a variety of research topics that we aim to pursue in future work.

\begin{acknowledgement}
	The author expresses his gratitude to Prof. Rainer Helmig, Prof. Ivan Yotov, Dennis Gl\"aser, and Kilian Weishaupt for valuable discussions on closely related topics.
\end{acknowledgement}

\bibliographystyle{siam}
\bibliography{biblio}

\begin{thebibliography}{10}

\bibitem{arbogast2000mixed}
{\sc T.~Arbogast, L.~C. Cowsar, M.~F. Wheeler, and I.~Yotov}, {\em Mixed finite
  element methods on nonmatching multiblock grids}, SIAM Journal on Numerical
  Analysis, 37 (2000), pp.~1295--1315.

\bibitem{armentano2019unified}
{\sc M.~G. Armentano and M.~L. Stockdale}, {\em A unified mixed finite element
  approximations of the {S}tokes--{D}arcy coupled problem}, Computers \&
  Mathematics with Applications, 77 (2019), pp.~2568--2584.

\bibitem{baerland2019multigrid}
{\sc T.~B{\ae}rland, M.~Kuchta, and K.-A. Mardal}, {\em Multigrid methods for
  discrete fractional {S}obolev spaces}, SIAM Journal on Scientific Computing,
  41 (2019), pp.~A948--A972.

\bibitem{bernardi1985analysis}
{\sc C.~Bernardi and G.~Raugel}, {\em Analysis of some finite elements for the
  {S}tokes problem}, Mathematics of Computation,  (1985), pp.~71--79.

\bibitem{boffi2013mixed}
{\sc D.~Boffi, F.~Brezzi, and M.~Fortin}, {\em Mixed finite element methods and
  applications}, vol.~44, Springer, 2013.

\bibitem{boon2018robust}
{\sc W.~M. Boon, J.~M. Nordbotten, and I.~Yotov}, {\em Robust discretization of
  flow in fractured porous media}, SIAM Journal on Numerical Analysis, 56
  (2018), pp.~2203--2233.

\bibitem{brezzi1985two}
{\sc F.~Brezzi, J.~Douglas, and L.~D. Marini}, {\em Two families of mixed
  finite elements for second order elliptic problems}, Numerische Mathematik,
  47 (1985), pp.~217--235.

\bibitem{cao2011robin}
{\sc Y.~Cao, M.~Gunzburger, X.~He, and X.~Wang}, {\em {Robin}--{Robin} domain
  decomposition methods for the steady-state {S}tokes--{D}arcy system with the
  {B}eavers--{J}oseph interface condition}, Numerische Mathematik, 117 (2011),
  pp.~601--629.

\bibitem{discacciati2004domain}
{\sc M.~Discacciati}, {\em Domain decomposition methods for the coupling of
  surface and groundwater flows}, PhD thesis, 2004.

\bibitem{discacciati2005iterative}
\leavevmode\vrule height 2pt depth -1.6pt width 23pt, {\em Iterative methods
  for {S}tokes/{D}arcy coupling}, in Domain decomposition methods in science
  and engineering, Springer, 2005, pp.~563--570.

\bibitem{discacciati2018optimized}
{\sc M.~Discacciati and L.~Gerardo-Giorda}, {\em Optimized schwarz methods for
  the {S}tokes--{D}arcy coupling}, IMA Journal of Numerical Analysis, 38
  (2018), pp.~1959--1983.

\bibitem{discacciati2009navier}
{\sc M.~Discacciati and A.~Quarteroni}, {\em Navier-{S}tokes/{D}arcy coupling:
  Modeling, analysis, and numerical approximation}, Revista Matem{\'a}tica
  Complutense, 22 (2009), pp.~315--426.

\bibitem{discacciati2007robin}
{\sc M.~Discacciati, A.~Quarteroni, and A.~Valli}, {\em Robin--robin domain
  decomposition methods for the {S}tokes--{D}arcy coupling}, SIAM Journal on
  Numerical Analysis, 45 (2007), pp.~1246--1268.

\bibitem{evans2010partial}
{\sc L.~Evans}, {\em Partial Differential Equations}, Graduate studies in
  mathematics, American Mathematical Society, 2010.

\bibitem{fetzer2017a}
{\sc T.~Fetzer, C.~Gr{\"u}ninger, B.~Flemisch, and R.~Helmig}, {\em On the
  conditions for coupling free flow and porous-medium flow in a finite volume
  framework}, in International Conference on Finite Volumes for Complex
  Applications, Springer, 2017, pp.~347--356.

\bibitem{galvis2007balancing}
{\sc J.~Galvis and M.~Sarkis}, {\em Balancing domain decomposition methods for
  mortar coupling {S}tokes-{D}arcy systems}, in Domain decomposition methods in
  science and engineering XVI, Springer, 2007, pp.~373--380.

\bibitem{ganderderivation}
{\sc M.~J. Gander and T.~Vanzan}, {\em On the derivation of optimized
  transmission conditions for the {S}tokes-{D}arcy coupling}, Domain
  Decomposition Methods in Science and Engineering XXV,  (2019).

\bibitem{gatica2011analysis}
{\sc G.~Gatica, R.~Oyarz{\'u}a, and F.-J. Sayas}, {\em Analysis of fully-mixed
  finite element methods for the {S}tokes-{D}arcy coupled problem}, Mathematics
  of Computation, 80 (2011), pp.~1911--1948.

\bibitem{gatica2009conforming}
{\sc G.~N. Gatica, S.~Meddahi, and R.~Oyarz{\'u}a}, {\em A conforming mixed
  finite-element method for the coupling of fluid flow with porous media flow},
  IMA Journal of Numerical Analysis, 29 (2009), pp.~86--108.

\bibitem{iliev2004a}
{\sc O.~Iliev and V.~Laptev}, {\em On numerical simulation of flow through oil
  filters}, Computing and Visualization in Science, 6 (2004), pp.~139--146.

\bibitem{karper2009unified}
{\sc T.~Karper, K.-A. Mardal, and R.~Winther}, {\em Unified finite element
  discretizations of coupled {D}arcy--{S}tokes flow}, Numerical Methods for
  Partial Differential Equations, 25 (2009), pp.~311--326.

\bibitem{kuchta2016preconditioners}
{\sc M.~Kuchta, M.~Nordaas, J.~C. Verschaeve, M.~Mortensen, and K.-A. Mardal},
  {\em Preconditioners for saddle point systems with trace constraints coupling
  2{D} and 1{D} domains}, SIAM Journal on Scientific Computing, 38 (2016),
  pp.~B962--B987.

\bibitem{layton2002coupling}
{\sc W.~J. Layton, F.~Schieweck, and I.~Yotov}, {\em Coupling fluid flow with
  porous media flow}, SIAM Journal on Numerical Analysis, 40 (2002),
  pp.~2195--2218.

\bibitem{lions2012non}
{\sc J.~L. Lions and E.~Magenes}, {\em Non-homogeneous boundary value problems
  and applications}, vol.~1, Springer Science \& Business Media, 2012.

\bibitem{logg2012automated}
{\sc A.~Logg, K.-A. Mardal, and G.~Wells}, {\em Automated solution of
  differential equations by the finite element method: The FEniCS book},
  vol.~84, Springer Science \& Business Media, 2012.

\bibitem{mardal2011preconditioning}
{\sc K.-A. Mardal and R.~Winther}, {\em Preconditioning discretizations of
  systems of partial differential equations}, Numerical Linear Algebra with
  Applications, 18 (2011), pp.~1--40.

\bibitem{masson2016a}
{\sc R.~Masson, L.~Trenty, and Y.~Zhang}, {\em Coupling compositional liquid
  gas {D}arcy and free gas flows at porous and free-flow domains interface},
  Journal of Computational Physics, 321 (2016), pp.~708--728.

\bibitem{mosthaf2011a}
{\sc K.~Mosthaf, K.~Baber, B.~Flemisch, R.~Helmig, A.~Leijnse, I.~Rybak, and
  B.~Wohlmuth}, {\em A coupling concept for two-phase compositional
  porous-medium and single-phase compositional free flow}, Water Resour. Res,
  47 (2011).

\bibitem{nordbotten2019unified}
{\sc J.~M. Nordbotten, W.~M. Boon, A.~Fumagalli, and E.~Keilegavlen}, {\em
  Unified approach to discretization of flow in fractured porous media},
  Computational Geosciences, 23 (2019), pp.~225--237.

\bibitem{quarteroni1999domain}
{\sc A.~Quarteroni, A.~Valli, and P.~Valli}, {\em Domain Decomposition Methods
  for Partial Differential Equations}, Numerical Mathematics and Scie,
  Clarendon Press, 1999.

\bibitem{raviart1977mixed}
{\sc P.-A. Raviart and J.-M. Thomas}, {\em A mixed finite element method for
  2-nd order elliptic problems}, in Mathematical aspects of finite element
  methods, Springer, 1977, pp.~292--315.

\bibitem{riviere2005locally}
{\sc B.~Rivi{\`e}re and I.~Yotov}, {\em Locally conservative coupling of
  {S}tokes and {D}arcy flows}, SIAM Journal on Numerical Analysis, 42 (2005),
  pp.~1959--1977.

\bibitem{rybak2016mathematical}
{\sc I.~Rybak}, {\em Mathematical modeling of coupled free flow and porous
  medium systems}, {H}abilitation, University of Stuttgart, 2016.

\bibitem{rybak2015a}
{\sc I.~Rybak, J.~Magiera, R.~Helmig, and C.~Rohde}, {\em Multirate time
  integration for coupled saturated/unsaturated porous medium and free flow
  systems}, Computational Geosciences, 19 (2015), pp.~299--309.

\bibitem{saad1986gmres}
{\sc Y.~Saad and M.~H. Schultz}, {\em Gmres: A generalized minimal residual
  algorithm for solving nonsymmetric linear systems}, SIAM Journal on
  scientific and statistical computing, 7 (1986), pp.~856--869.

\bibitem{schneider2020coupling}
{\sc M.~Schneider, K.~Weishaupt, D.~Gl{\"a}ser, W.~M. Boon, and R.~Helmig},
  {\em Coupling staggered-grid and {MPFA} finite volume methods for free
  flow/porous-medium flow problems}, Journal of Computational Physics, 401
  (2020), p.~109012.

\end{thebibliography}

\end{document}